\documentclass[12pt,a4paper,twoside,reqno]{amsart}

\usepackage{tikz}
\usepackage{slashed}
\usepackage{amsmath,amscd}
\usepackage{amssymb,amsfonts,mathrsfs}
\usepackage[driver=pdftex,margin=3cm,heightrounded=true,centering]{geometry}
\usepackage{JK}
\usepackage[colorlinks=true,linkcolor=blue,citecolor=blue]{hyperref}

\newtheorem{thm}{Theorem}[section]

\usepackage{graphicx}

\begin{document}
\title{Topological invariance of the homological index}

\author{Alan Carey}
\address{Mathematical Sciences Institute,
Australian National University,
Canberra 0200, Australia}
\email{alan.carey@anu.edu.au}

\author{Jens Kaad}
\address{International School of Advanced Studies (SISSA),
Via Bonomea 265,
34136 Trieste,
Italy}
\email{jenskaad@hotmail.com}

%
%
%
\thanks{The first author thanks the Alexander von Humboldt Stiftung and colleagues at the University of M\"unster and acknowledges the support of the Australian Research Council. The second author was partly supported by the Fondation Sciences Math\'ematiques de Paris (FSMP) and by a public grant overseen by the French National Research Agency (ANR) as part of the ``Investissements d'Avenir'' program (reference: ANR-10-LABX-0098). Both authors are very appreciative of the support offered by the Erwin Schr\"odinger Institute where much of this research was carried out. We are also grateful for the advice of Joachim Cuntz, Harald Grosse and Fritz Gesztesy while this investigation was proceeding.}
\subjclass[2010]{19K56; 46L80, 47A55}
\keywords{Index theory, cyclic theory, perturbations} 

\begin{abstract}
R. W. Carey and J. Pincus in \cite{CaPi:IOG} proposed an index theory for non-Fredholm bounded operators $T$ on a separable Hilbert space $\C H$ such that $TT^*-T^*T$ is in the trace class.
We showed in  \cite{CaGrKa:AOS} using Dirac-type operators acting on sections of bundles over $\mathbb R^{2n}$ that we could construct bounded operators $T$ satisfying the more general condition
that $(1-TT^*)^{n}- (1-T^*T)^{n}$ is trace class. We proposed there a `homological index' for these Dirac-type operators given by
${\rm Tr}((1-TT^*)^{n}- (1-T^*T)^{n})$. In this paper we show that the index introduced  in \cite{CaGrKa:AOS}  represents the result of a paring between a cyclic homology theory for the algebra generated by $T$ and $T^*$ and its dual cohomology theory. This leads us to establish the homotopy invariance of our homological index (in the sense of cyclic theory).
We are then able to define in a very general fashion a homological index for certain unbounded operators and prove invariance of this index under a class of unbounded perturbations.
\end{abstract}

\maketitle
\tableofcontents
\section{Introduction}

\subsection{Background}
An `index
theory' for non-Fredholm operators was commenced  some time ago by R. W. Carey and J. Pincus in \cite{CaPi:IOG}, and F. Gesztesy and B. Simon in \cite{GeSi:TIW}.
In both of these papers the index is expressed in terms of the Krein spectral shift function from scattering theory. In the former paper the starting point is an operator $T$ on a separable Hilbert space
$\mathcal H$ with the property that the commutator $TT^*-T^*T$ is trace class.
In the latter paper the problem is stated for unbounded operators motivated by examples in \cite{BGGSS:WKS}.
The passage from the unbounded picture to the bounded one is straightforward and is explained below (see also our
  companion paper \cite{CaGrKa:AOS} and \cite{CGPST:WSS}).

The main point of the companion paper \cite{CaGrKa:AOS} was to demonstrate the existence of a class of non-trivial examples to which the general framework described here applies. These examples are Euclidean Dirac type operators on $\mathbb R^{2n}$. They illustrate the appropriate generalisation of the Carey-Pincus framework
to the case where one replaces the trace ideal by other Schatten ideals. 

The primary purpose of the discussion below is to explain, for the bounded picture, a homological formulation of an index theory for non-Fredholm operators where we  impose a modification of the  Carey-Pincus trace class commutator condition. This also entails a discussion of the invariance properties of our
homological index.  We then provide conditions on perturbations under which the homological index for unbounded operators
is unchanged. There is a relationship to the perturbation invariance result that appears in  \cite{BGGSS:WKS}.

Our conditions apply to the examples in \cite{CaGrKa:AOS}. There we observed that  the trace class commutator condition of Carey-Pincus is  relevant to low dimensional manifolds but does not apply in higher dimensions. There is more than one way to generalise the 
Carey-Pincus theory. We believe that the homological development that we provide here is natural
from the point of view of the  examples in higher dimensions described in \cite{CaGrKa:AOS}. 

The index studied in \cite{CaPi:IOG}, and \cite{GeSi:TIW} is not invariant under compact perturbations and hence has no relationship to K-theory.  In this paper we show how to use cyclic homology as a substitute for K-theory in the sense of expressing the numerical  index as the outcome of a pairing of cohomology and homology theories.  Specifically, we define our  `homological index' as a functional on homology groups of a bicomplex for the algebra generated by $T$ and $T^*$.
This bicomplex uses a relative homology construction and  is adapted from the usual $(b,B)$ complex of  cyclic theory
\cite{Lod:CH}. 

Our main theorem  establishes the homotopy
invariance, in the sense of  cyclic homology, of our homological index. This enables us to then understand which perturbations  of the Dirac type
operators in \cite{CaGrKa:AOS}  leave the index introduced there invariant.
\subsection{Outline of our approach}
Our generalisation of the Carey-Pincus work 
begins with a bounded operator $T$ on $\mathcal H$
such that 
\begin{equation}\label{condition}
(1-T^*T)^n-(1-TT^*)^n
\end{equation}
is in the trace class. 
For $n=1$ this condition reduces to the trace class commutator condition.  For $n>1$  we have:

\begin{lemma}{\label{Suklemma}}
If $(1-TT^*)^{n}- (1-T^*T)^{n}$ is trace class then $T$ and $T^*$ commute modulo the Schatten class $\sL^n(\C H)$.
\end{lemma}

\begin{proof}
This result is a Corollary of \cite[Theorem 16]{PoSu:DIS}.
 \end{proof}

The converse to Lemma \ref{Suklemma} appears to require additional side conditions.

Our next step is to introduce  certain homology groups  of  the *-algebra generated by $T$. 
We present this homological approach in a more abstract framework where we are given an algebra $\mathcal A$ with two ideals $\mathcal I$ and $\mathcal J$ with $\mathcal J\subset\mathcal I$. Later $\mathcal J$ is chosen to be  the ideal generated by $(1-T^*T)^n- (1-TT^*)^n$ while $\mathcal I$ is the smallest ideal containing $(1-T^*T)^n$ and $(1-TT^*)^n$.
Our innovation in this paper is to introduce a bicomplex for the algebra $\mathcal A$ by using the ideals $\mathcal I$ and $\mathcal J$. The homology theory of our bicomplex has a dual cohomology theory
and the pairing between the two,  in the concrete situation of operators on Hilbert space, produces the real number ${\rm{Tr}}((1-T^*T)^n-(1-TT^*)^n)$
 that we call the homological index. We then establish the homotopy invariance of this pairing in the sense of  cyclic homology.

The second part of the paper applies this theory to some examples of unbounded  operators that can be mapped to a pair $T, T^*$ satisfying a Schatten class condition as above. 
These examples  are of two types. The first is motivated by \cite{CGPST:WSS} where one starts from a spectral flow problem, and then using the usual doubling trick 
constructs a $\mathbb Z_2$ graded space and an associated `index problem'. This index problem can be mapped to the bounded picture producing a pair $T,T^*$ as above.

A second source of examples is studied in \cite{CaGrKa:AOS}. They arise from taking a Dirac operator acting on $L^2$-sections of the spin bundle over $\rr^{2n}$ and coupling it to a connection. If we write this operator as $\mathcal D$ 
then we pass to the bounded picture using the map $\mathcal D\to \mathcal D(1+\mathcal D^2)^{-1/2}$. Using the natural grading on the $L^2$ sections afforded by the Clifford algebra we can construct
a pair of operators $T, T^*$ that satisfy the Schatten class condition above. 

By scaling $\mathcal D$, that is, replacing it by $\mu^{-1/2} \mathcal D$ we obtain a one parameter family of pairs $T_\mu, T^*_\mu$. In \cite{CaGrKa:AOS} we studied the scaling limit as $\mu\to \infty$
of the homological index
$${\rm{Tr}}((1-T_\mu^*T_\mu)^n-(1-T_\mu T_\mu^*)^n).$$
This limit is referred to as the `anomaly' in the mathematical physics literature. On the other hand we remark that it is the scaling limit as $\mu\to 0$ that is studied in \cite{GeSi:TIW} and which motivated much of the subsequent work.\footnote{This limit is termed the Witten index in \cite{GeSi:TIW}
after \cite{Wi:CSB}. It was Gesztesy-Simon who discovered the connection between Witten's ideas and
the spectral shift function and hence to Carey-Pincus. There is an extensive literature on the Witten index and supersymmetric quantum mechanics. As we do not pursue these ideas in this paper we refer to 
\cite{CGPST:WSS} for more detail on this history.}

\subsection{The general formalism for unbounded operators}\label{ss:forunbope}
%
%
%


The unbounded operators we consider in \cite{CaGrKa:AOS} and which arise in other contexts such as  \cite{CGPST:WSS} 
have the following structure.

First we double our Hilbert space setting $\C H^{(2)} := \C H \op \C H$. We let $\C D^+$ be a closed densely defined operator on $\C H$ and form the odd selfadjoint operator $\C D := \ma{cc}{0 & \C D^- \\ \C D^+ & 0}$, where $\C D^- = (\C D^+)^*$. We will study a class of perturbations of $\C D$ of the form
\[
\C D_A := \ma{cc}{
0 & \C D^- + A^- \\
\C D^+ + A^+ & 0
}
\]
where $A^+$ is an unbounded closed operator on $\C H$ (generally satisfying some side conditions so that the following manipulations are valid) and $A^- := (A^+)^*$. The connection to the homological index is via the mapping to bounded operators using the Riesz map
\[
\C D^+ + A^+ \mapsto T^+ = (\C D^+ + A^+)\big(1 + (\C D^-+ A^-)(\C D^+ + A^+) \big)^{-1/2}
\]
These bounded operators generate an algebra to which our homological theory applies. To see how this arises we note the identities
\[
\C D_A^2 = \ma{cc}{
(\C D^- + A^-)(\C D^+ + A^+) & 0 \\
0 & (\C D^+ + A^+)(\C D^- + A^-)
}
\]
and
\begin{equation}\label{eq:boures}
\begin{split}
& 1 - T^+ (T^+)^* = ( 1 + (\C D^++A^+) (\C D^-+A^-) )^{-1} \\
& 1 - (T^+)^* T^+ = ( 1 + (\C D^- +A^-)(\C D^+ +A^+))^{-1}
\end{split}
\end{equation}

We show in \cite{CaGrKa:AOS} that there is a natural class of Dirac-type operators on $\rr^{2n}$, $n \in \nn$ which fit the above framework and satisfy the extra condition:
\begin{equation}\label{eq:summ}
\begin{split}
(1 - (T^+)^* T^+)^n - (1 - T^+ (T^+)^*)^n \in \sL^1(\C H)
\end{split}
\end{equation}
where $\sL^1(\C H) \su \sL(\C H)$ denotes the ideal of trace class operators. The connection between the dimension of the underlying space $\rr^{2n}$ and the condition \eqref{eq:summ} 
is not evident in the earlier work \cite{BGGSS:WKS} but is natural from the point of view of spectral
and noncommutative geometry. 
\footnote{We remark however that, whereas the notion of spectral triple is prominent in noncommutative geometry, we do not have this additional structure in our approach. Rather, as in \cite{CaPi:IOG}, it is the C$^*$-algebra generated by $T^+$ that is being investigated by spectral methods here.
The reader may recall from \cite[Introduction to Chapter III]{Con:NCG} that the early work of Carey-Pincus provided inspiration for the initial development of noncommutative geometry. This paper shows that the later work of
\cite{CaPi:IOG} points to a new direction that exploits other tools from noncommutative geometry.}

\subsection{The main results and outline of the paper}

The paper begins, in Section \ref{s:precycthe}, with an outline of the relevant notions from cyclic theory and an explanation of the construction of
the bicomplex needed for the homological index. In Section \ref{s:seqcyc} we show how the operators of interest in the definition of the homological index form part of a 2-cycle in this bicomplex.
Then in Section \ref{s:tracoc} we show that the homological index coincides with the numerical pairing
of our homology theory with its dual.

Section \ref{s:invarpro} contains our topological invariance results in the form of theorems \ref{t:hominv} and \ref{t:hominvII}. The latter in particular investigates when a pair $T_0^+,T_1^+$ that are joined by a norm differentiable path $\{T_t^+:t\in [0,1]\}$ of operators have the same homological index. We formulate the conditions in terms of trace norm continuity properties of certain functions of $T_t^+$ and $(T_t^+)^*$ and their derivatives.

The interesting examples of the homological index arise from differential operators. 
So we begin with unbounded operators $\C D^+$ and the auxiliary structure described in the previous Subsection.
In Section \ref{s:homindunb} we define the homological index of these unbounded operators. In Section \ref{s:exiunbhom} we work in the framework of Subsection \ref{ss:forunbope} and show, under certain constraints on the perturbation $A^+$, that the homological index exists for $\C D^+$ exists if and only if it exists for the perturbation $\C D^++A^+$. Finally in Section \ref{s:invunbhom} we establish conditions under which the homological indices of $\C D^+$  and  $\C D^++A^+$ are the same. In the first Appendix we sketch how the Theorem of Section \ref{s:invunbhom} applies to the Dirac type operators introduced in \cite{CaGrKa:AOS}. In the second Appendix we have gathered some general results on perturbations of unbounded operators.

\section{Preliminaries on cyclic theory}\label{s:precycthe}
In this Section we collect the concepts from cyclic theory needed for the paper.
Throughout this Section $\C A$ will be a unital algebra over the complex numbers. Our general references for cyclic theory are the books of Connes and Loday, \cite[Chapter III]{Con:NCG} and \cite{Lod:CH}.

\subsection{Cyclic homology}
For each $k \in \nn \cup \{0\}$, let $C_k(\C A) := \C A^{\ot (k+1)}$. Define the Hochschild boundary $b : C_k(\C A) \to C_{k-1}(\C A)$ by
\[
b : x_0 \olo x_k \mapsto 
\sum_{i=0}^{k-1} (-1)^k x_0 \olo x_i x_{i+1} \olo x_k + (-1)^k x_k x_0 \ot x_1 \olo x_{k-1}
\]
where $C_{-1}(\C A) := \{0\}$. Define the cyclic operator $t : C_k(\C A) \to C_k(\C A)$ by
\[
t : x_0 \olo x_k \mapsto (-1)^k x_k \ot x_0 \olo x_{k-1}
\]
The norm operator $N : C_k(\C A) \to C_k(\C A)$ is then $N := 1 + t \plp t^k$. The extra degeneracy $s : C_k(\C A) \to C_{k+1}(\C A)$ is defined by
\[
s : x_0 \olo x_k \mapsto 1 \ot x_0 \olo x_k.
\]
The Connes boundary $B : C_k(\C A) \to C_{k+1}(\C A)$ is then given by $B := (1-t)sN$.

The Connes-bicomplex $\C B_{**}(\C A)$ is defined by the diagram
\[
\begin{CD}
@V{b}VV  @V{b}VV @V{b}VV \\
C_2(\C A) @<<{B}< C_1(\C A) @<<{B}< C_0(\C A) \\ 
@V{b}VV @V{b}VV \\
C_1(\C A) @<<{B}< C_0(\C A) \\ 
@V{b}VV \\
C_0(\C A)
\end{CD}
\]
The cyclic homology $HC_*(\C A)$ of $\C A$ is defined as the homology of the totalization of the Connes-bicomplex. The chains are thus given by
\[
\T{Tot}_{2k}(\C B(\C A)) = \bop_{m = 0}^k C_{2m}(\C A) \q \T{and} \q
\T{Tot}_{2k+1}(\C B(\C A)) = \bop_{m = 0}^k C_{2m+1}(\C A)
\]
in even and odd degrees respectively. The boundary map is given by
\[
(b + B) : \sum_{m=0}^k \xi_{2m} e_{2m} \mapsto 
\sum_{m=0}^{k-1}\big( b(\xi_{2m + 2}) + B(\xi_{2m}) \big) e_{2m+1}
\]
on even chains. And by a similar formula on odd chains. 

\subsection{Cyclic cohomology}
The cyclic cohomology of $\C A$ is obtained by applying the contravariant functor $\T{Hom}_{\cc}(\cd,\cc)$ to the homological constructions in the last subsection. More explicitly, for each $k \in \nn \cup \{0\}$, let $C^k(\C A) := \T{Hom}_{\cc}\big(C_k(\C A),\cc \big)$. The Hochschild coboundary and the Connes coboundary are then defined by
\[
\begin{split}
& b : C^k(\C A) \to C^{k+1}(\C A) \q b : \rho \mapsto \rho \ci b \q \T{and} \\
& B : C^k(\C A) \to C^{k-1}(\C A) \q B : \rho \mapsto \rho \ci B
\end{split}
\]
respectively. This gives rise to the Connes-bicomplex $B^{**}(\C A)$ defined by the diagram
\[
\begin{CD}
@A{b}AA  @A{b}AA @A{b}AA \\
C^2(\C A) @>>{B}> C^1(\C A) @>>{B}> C^0(\C A) \\
@A{b}AA @A{b}AA \\
C^1(\C A) @>>{B}> C^0(\C A) \\
@A{b}AA \\
C^0(\C A)
\end{CD}
\]
The cyclic cohomology $HC^*(\C A)$ is the cohomology of the totalization of the bicomplex $\C B^{**}(\C A)$.

\subsection{Relative cyclic homology}\label{s:relcychom}
Let $\C I \su \C A$ be an ideal in $\C A$. The quotient map $q : \C A \to \C A/\C I$ induces a surjective map of bicomplexes
\[
q : \C B_{**}(\C A) \to \C B_{**}(\C A/\C I).
\]
The relative cyclic bicomplex is defined as the kernel of the chain map $q$. This bicomplex is denoted by $\C B_{**}(\C A,\C I)$. The homology of its totalization is the relative cyclic homology $HC_*(\C A,\C I)$. By fundamental homological algebra there is an associated long exact sequence
\[
\begin{CD}
\cdots \, \, HC_{*+1}(\C A/\C I) @>{\pa}>> HC_*(\C A,\C I) @>{i}>> 
HC_*(\C A) @>{q}>> HC_*(\C A/\C I) @>{\pa}>> \ldots
\end{CD}
\]

The bicomplex $\C B_{**}(\C A,\C I)$ can be described explicitly: for each $k \in \nn \cup \{0\}$, define the subspace
\[
C_k(\C A,\C I) := \C I \ot \C A \olo \C A \plp \C A \olo \C A \ot \C I,
\]
where the algebra $\C A$ appears precisely $k$ times. We remark that the Hochschild boundary and the Connes boundary restrict to boundary maps
\[
b : C_k(\C A,\C I) \to C_{k-1}(\C A,\C I) \q \T{and} \q 
B : C_k(\C A,\C I) \to C_{k+1}(\C A,\C I).
\]
The bicomplex $\C B_{**}(\C A,\C I)$ is then given by
\[
\begin{CD}
@V{b}VV  @V{b}VV @V{b}VV \\
C_2(\C A,\C I) @<<{B}< C_1(\C A,\C I) @<<{B}< C_0(\C A,\C I) \\ 
@V{b}VV @V{b}VV \\
C_1(\C A,\C I) @<<{B}< C_0(\C A,\C I) \\ 
@V{b}VV \\
C_0(\C A,\C I)
\end{CD}
\]

\subsection{Relative cyclic cohomology}
Let $\C I \su \C A$ be an ideal in $\C A$. The description of the cyclic homology of the quotient algebra $\C A/\C I$ by a mapping cone construction can also be dualized by applying the contravariant functor $\T{Hom}_\cc(\cd,\cc)$. This yields a convenient description of the cyclic cohomology of the quotient $\C A/\C I$ which we will apply in the subsequent sections.

For each $k \in \nn \cup \{0\}$, let $C^k(\C A,\C I) := \T{Hom}_{\cc}\big(C_k(\C A,\C I),\cc\big)$. The relative cyclic cohomology 
of the pair $\C I \su \C A$ is obtained as the cohomology of the totalization of the bicomplex $\C B^{**}(\C A,\C I)$ defined by the diagram
\[
\begin{CD}
@A{b}AA  @A{b}AA @A{b}AA \\
C^2(\C A,\C I) @>>{B}> C^1(\C A,\C I) @>>{B}> C^0(\C A,\C I) \\
@A{b}AA @A{b}AA \\
C^1(\C A,\C I) @>>{B}> C^0(\C A,\C I) \\
@A{b}AA \\
C^0(\C A,\C I).
\end{CD}
\]

\subsection{The glued complex}
Let $X$ and $Y$ be first quadrant bicomplexes. The vertical and horizontal boundaries are denoted by $d^v_X, d^h_X$ and $d^v_Y, d^h_Y$ respectively. We will use the convention that $d^v_X d^h_X + d^h_X d^v_X = 0$ and similarly that $d^v_Y d^h_Y + d^h_Y d^v_Y = 0$. Let $X_{*0}$ and $Y_{*0}$ denote the $0^{\T{th}}$ column of $X$ and $Y$ respectively. {\color{black} We remark} that both $X_{*0}$ and $Y_{*0}$ are chain complexes when equipped with their vertical boundaries. Suppose that we have a chain map $\al : X_{*0} \to Y_{*0}$ of degree $1$. This means that $\al : X_{k,0} \to Y_{k+1,0}$ satisfies the relation $d^Y \al + \al d^X = 0$.

\begin{dfn}
The \emph{concatenation} of $X$ and $Y$ along $\al$ is the bicomplex $Y \coprod_\al X$ which in degree $(k,m)$ is given by
\[
(Y \coprod_\al X)_{k,m} := \fork{ccc}{
Y_{k,0} & \T{for} & m = 0 \\
X_{(k-1),(m-1)} & \T{for} & m \in \{1,2,\ldots\}.
}
\]
The horizontal boundary is given by $d^h(\al) : (Y \coprod_\al X)_{k,m} \to (Y \coprod_\al X)_{k,(m-1)}$,
\[
d^h(\al) := \fork{ccc}{
\al & \T{for} & m = 1 \\
d^h_X & \T{for} & m \geq 2
}.
\]
The vertical boundary is given by $d^v(\al) : (Y \coprod_\al X)_{k,m} \to (Y \coprod_\al X)_{(k-1),m}$,
\[
d^v(\al) := \fork{ccc}{
d^v_Y & \T{for} & m = 0 \\
d^v_X & \T{for} & m \geq 1
}.
\]
\end{dfn}

\subsection{The two ideal case}
Let $\C J \su \C I \su \C A$ be two ideals in the unital $\cc$-algebra $\C A$. We then have the two first quadrant bicomplexes $\C B_{**}(\C I,\C A)$ and $\C B_{**}(\C J,\C A)$. 

The $0^{\T{th}}$ column of $\C B_{**}(\C I,\C A)$ is given by $\C B_{*0}(\C I,\C A) = C_*(\C I,\C A)$. Thus, in degree $k \in \nn \cup \{0\}$ we have the chains $C_k(\C I,\C A)$, where the boundary map $b : C_k(\C I,\C A) \to C_{k-1}(\C I,\C A)$ is induced by the Hochschild boundary operator.

The $0^{\T{th}}$ column of $\C B_{**}(\C J,\C A)$ is related to the $0^{\T{th}}$ column of $\C B_{**}(\C I,\C A)$ by the Connes boundary $B : C_*(\C J,\C A) \to C_*(\C I,\C A)$, which is a chain map of degree $1$.

The concatenation of $\C B_{**}(\C J,\C A)$ and $\C B_{**}(\C I,\C A)$ along $B$ can be represented by the diagram
\[
\begin{CD}
@V{b}VV  @V{b}VV @V{b}VV \\
C_2(\C A,\C I) @<<{B}< C_1(\C A,\C J) @<<{B}< C_0(\C A,\C J) \\
@V{b}VV @V{b}VV \\
C_1(\C A,\C I) @<<{B}< C_0(\C A,\C J) \\
@V{b}VV \\
C_0(\C A,\C I)
\end{CD}
\]
We will use the notation $\C B_{**}(\C I,\C J,\C A) := \C B_{**}(\C I,\C A) \coprod_B \C B_{**}(\C J,\C A)$ for this bicomplex.

\begin{dfn}\label{S-op}
The totalization of the concatenation we write as $\C B_*(\C I,\C J,\C A)$ or $ \T{Tot}_*\big(\C B(\C I,\C J,\C A)\big)$ and it comes equipped with a periodicity operator, denoted 
$$S : \C B_*(\C I,\C J,\C A) \to \C B_{*-2}(\C J,\C A)$$ and defined by $S(x_n,x_{n-2},\ldots) = (x_{n-2},x_{n-4},\ldots)$.
\end{dfn}

%

\section{A sequence of cycles}\label{s:seqcyc}
Let $\C A$ denote the free unital algebra over $\cc$ on two generators $x,y\in \C A$. The elements $v := 1 -xy$ and $w := 1 - yx$ in $\C A$ will play a special role. Let $n \in \nn$ be fixed. Let $\C I_n \su \C A$ denote the smallest ideal which contains both $v^n$ and $w^n$ and let $\C J_n \su \C A$ denote the ideal generated by the element $w^n - v^n$.

Define the element $\om_n := x + vx \plp v^{n-1}x$.

\begin{lemma}\label{l:baside}
\[
y \cd \om_n = 1 - w^n \q \om_n \cd y = 1 - v^n \q v \cd \om_n = \om_n \cd w
\]
\end{lemma}
\begin{proof}
Remark first that $y \cd v = w \cd y$. This implies that 
\[
y \cd \om_n = (1 + w \plp w^{n-1})yx = (1 + w \plp w^{n-1})(1 - w) = 1 - w^n 
\]
Similarly,
\[
\om_n \cd y = (1 + v \plp v^{n-1})xy = (1 + v \plp v^{n-1})(1 - v) = 1 - v^n
\]
This proves the first two identities.
The last identity follows from the computation
\[
v \cd \om_n = v x + v^2 x \plp v^n x = x w + v x w \plp v^{n-1} x w = \om_n \cd w
\]
\end{proof}

\begin{dfn}\label{chains}
We define the chains $\gamma_2$ and $\gamma_0$ by
\[
\begin{split}
\ga_2 & := -w^n \ot y \ot \om_n + v^n \ot \om_n \ot y + \om_n \ot w^n \ot y + \om_n \ot y \ot v^n \\ 
& \q - 2 \cd (w^n - v^n) \ot 1 \ot 1 \in C_2(\C I_n,\C A) \\
\ga_0 & := w^n - v^n \in \C J_n
\end{split}
\]
\end{dfn}

\begin{prop}\label{p:twococ}
The $2$-chain $\ga:= (\ga_2,\ga_0)$ is a $2$-cycle in the concatenation $\C B_*(\C I_n,\C J_n,\C A)$. 
\end{prop}
\begin{proof}
The Hochschild boundary of $\ga_2$ is given by
\[
\begin{split}
b(\ga_2) & = -w^n \cd y \ot \om_n + w^n \ot y\cd \om_n - \om_n \cd w^n \ot y \\
& \q + v^n \cd \om_n \ot y - v^n \ot \om_n \cd y + y \cd v^n \ot \om_n \\
& \q - \om_n \cd w^n \ot y + \om_n \ot w^n \cd y - y \cd \om_n \ot w^n \\
& \q + \om_n \cd y \ot v^n - \om_n \ot y \cd v^n + v^n \cd \om_n \ot y \\
& \q - 2 \cd (w^n - v^n) \ot 1 \\
& = w^n \ot y\cd \om_n - v^n \ot \om_n \cd y - y \cd \om_n \ot w^n + \om_n \cd y \ot v^n \\
& \q - 2 \cd (w^n - v^n) \ot 1 \\
& = -1 \ot (w^n - v^n) - (w^n - v^n) \ot 1,
\end{split}
\]
where the second identity uses the last identity of Lemma \ref{l:baside} and the third identity uses the first two identities of Lemma \ref{l:baside}.

The Connes boundary of $\ga_0$ is given by
\[
B(w^n - v^n) = 1 \ot (w^n - v^n) + (w^n - v^n) \ot 1 = -b(\ga_2).
\]
This proves the proposition.
\end{proof}

We remark that the application of the periodicity operator introduced in Definition \ref{S-op}, $S : \C B_*(\C I_n,\C J_n,\C A) \to \C B_*(\C J_n,\C A)$, to the cycle $\ga$ gives the difference $\ga_0 = w^n - v^n \in \C J_n$.


\section{The trace cocycle and the homological pairing}\label{s:tracoc}
\subsection{Definition of the trace cocycle}
Let $\C H$ be a Hilbert space. Let $\sL(\C H)$ denote the bounded operators on $\C H$ and let $\sL^1(\C H)$ denote the operators of trace class. The operator trace is denoted by $\T{Tr} : \sL^1(\C H) \to \cc$. The trace norm on $\sL^1(\C H)$ is denoted by $\|\cd \|_1 : \sL^1(\C H) \to [0,\infty)$, thus $\|x\|_1 = \T{Tr}(|x|)$, for all $x \in \sL^1(\C H)$. The operator norm on $\sL(\C H)$ is denoted by $\|\cd \| : \sL(\C H) \to [0,\infty)$.

Recall that the operator trace satisfies the identity $\T{Tr}(xy) = \T{Tr}(yx)$ for any pair of operators $x,y \in \sL(\C H)$ with $x\cd y$ and $y \cd x \in \sL^1(\C H)$. See \cite[Corollary 3.8]{Sim:TIA}.

Let $n \in \nn$. Define the vector subspace $\C K_n \su C_n\big( \sL(\C H) \big)$ as follows,
\[
\C K_n := \T{span}\big\{ x_0 \olo x_n \, | \, x_i \clc x_n \cd x_0 \clc x_{i-1} \in \sL^1(\C H) \, \T{ for all }\, i \in \{0,\ldots,n\} \, \big\}
\]
Notice that $\C K_0 = \sL^1(\C H)$. It follows by definition that the subspace $\C K_n$ is invariant under the cyclic operator $t : C_n\big(\sL(\C H) \big) \to C_n\big(\sL(\C H) \big)$.

Introduce the multiplication operator $M : C_n\big( \sL(\C H) \big) \to \sL(\C H)$ by defining  $M : x_0 \olo x_n \mapsto x_0 \clc x_n$. 

\begin{dfn}\label{d:tranorcha}
Define the norm $\|\cd \|_1 : \C K_n \to [0,\infty)$ by the formula,
\[
\begin{split}
\| \cd \|_1 : x \mapsto & \inf\Big\{ \sum_{j = 1}^m \Big( \sum_{i=0}^n \| M( t^i(x_0^j \olo x_n^j)) \|_1 + \|x_0^j\| \clc \|x_n^j\| \Big) \, \\
& \qqqq \Big| \, x = \sum_{j=1}^m x_0^j \olo x_n^j \Big\}.
\end{split}
\]
\end{dfn}

Let $K_n$ denote the completion of $\C K_n$ with respect to the norm $\|\cd \|_1$.
We remark that the cyclic operator induces an isometric isomorphism $t : K_n \to K_n$.

\begin{lemma}
The Hochschild boundary and the Connes boundary induce continuous boundary operators
\[
b : K_n \to K_{n-1} \q \T{and} \q B : K_n \to K_{n+1}
\]
for all $n \in \nn \cup \{0\}$.
\end{lemma}
\begin{proof}
Recall that the Hochschild boundary on $C_n\big( \sL(\C H) \big)$ is given by the sum $b = \sum_{k=0}^n (-1)^k d_k$, where the terms are defined by
\[
d_k(x_0 \olo x_n) := \fork{ccc}{
x_0 \olo x_k \cd x_{k+1} \olo x_n & \T{for} & k \in \{0,\ldots,n-1\} \\
x_n \cd x_0 \ot x_1 \olo x_{n-1} & \T{for} & k = n
}.
\]
Let $k \in \{0,\ldots,n\}$ be fixed. Notice then that
\[
\begin{split}
M t^i d_k 
& = \fork{ccc}{
M t^{i - n + 1 + k} d_{n-1} t^{n-1-k} & \T{for} & i \in \{n-k,\ldots,n-1\} \\
M d_{k+i} t^i & \T{for} & i \in \{0,\ldots,n-k-1\}
} \\
& = \fork{ccc}{
M t^{i+1} & \T{for} & i \in \{n-k,\ldots,n-1\} \\
M t^i & \T{for} & i \in \{0,\ldots,n-k-1\}
}
\end{split}
\]
It follows that $d_k$ induces a continuous linear map $d_k : \C K_n \to \C K_{n-1}$. This shows that the Hochschild boundary induces a continuous boundary operator $b : K_n \to K_{n-1}$.

To prove that the Connes boundary $B = (1-t)sN$ induces a continuous boundary operator $B : K_n \to K_{n+1}$ it suffices to note that the extra degeneracy induces a continuous linear map $s : \C K_n \to \C K_{n+1}$.
\end{proof}

Let $\C B_{**}(K)$ denote the bicomplex given by the diagram
\begin{equation}\label{eq:cyctracom}
\begin{CD}
@V{b}VV  @V{b}VV @V{b}VV \\
K_2 @<<{B}< K_1 @<<{B}< K_0 \\
@V{b}VV @V{b}VV \\
K_1 @<<{B}< K_0 \\ 
@V{b}VV \\
K_0
\end{CD}
\end{equation}
%

Let $K^n := K_n^*$ denote the Banach space dual of $K_n$.
We then have the bicomplex $\C B^{**}(K)$ defined by
\[
\begin{CD}
@A{b}AA  @A{b}AA @A{b}AA \\
K^2 @>>{B}> K^1 @>>{B}> K_0 \\
@A{b}AA @A{b}AA \\
K^1 @>>{B}> K^0 \\ 
@A{b}AA \\
K^0
\end{CD}
\]
where the coboundary operators are the duals of the boundary operators in \eqref{eq:cyctracom}. 
%

\begin{dfn}
The \emph{trace cocycle} $\T{Tr} \in K^0$ is given by the operator trace $\T{Tr} : K_0 = \sL^1(\C H) \to \cc$.
\end{dfn}

\noindent{\bf Remark}. The trace cocycle is a $0$-cocycle in $\T{Tot}\big(\C B^{**}(K) \big)$. Indeed, we have that $\T{Tr} \ci b = 0 : K_1 \to \cc$ since $\T{Tr}(x_0x_1) = \T{Tr}(x_1x_0)$ whenever $x_0 \ot x_1 \in \C K_1$, see \cite[Corollary 3.8]{Sim:TIA}.

\subsection{The homological index as a pairing}\label{s:homindpai}
Let $T \in \sL(\C H)$ be a bounded operator. Let $\pi : \C A \to \sL(\C H)$ be the unital algebra homomorphism given by $\pi : x \mapsto T$ and $\pi : y \mapsto T^*$. Recall here that $\C A$ denotes the free unital algebra over $\cc$ on the two generators $x$ and $y$.

\begin{dfn}\label{d:homindbou}
Suppose that there exists an $n \in \nn$ such that the difference $(1-T^*T)^n - (1-TT^*)^n \in \sL^1(\C H)$.
The \emph{homological index} of $T$ in degree $n \in \nn$ is defined as the number,
\[
\T{H-Ind}_n(T) := \T{Tr}\big((1-T^* T)^n - (1-T T^*)^n \big).
\]
\end{dfn}

Recall Definition \ref{chains} where we introduced  the cycle $\ga := (\ga_2,\ga_0) \in \C B_*(\C I_n,\C J_n,\C A)$. The unital algebra homorphism $\pi : \C A \to \sL(\C H)$ restricts to an algebra homomorphism $\pi : \C J_n \to \sL^1(\C H)$. 
 In particular, recalling here the Definition \ref{S-op} of the periodicity operator $S : \C B_2(\C I_{n},\C J_{n},\C A) \to \C B_0(\C J_{n},\C A)$ 
we have the cycle 

\[
(\pi \ci S)(\ga) = \pi(\ga_0) = (1-T^*T)^n - (1-T T^*)^n \in \C B_0(K).
\]

The following proposition is now clear.
\begin{prop}
The homological index coincides with the pairing in cyclic theory,
\[
\binn{ \T{Tr} ,(\pi \ci S)(\ga) } = \T{Tr}\big((1-T^* T)^{n} - (1-T T^*)^{n} \big) = \T{H-Ind}_n(T),
\]
where $\inn{\cd,\cd} : HC^0(K) \ti HC_0(K) \to \cc$.
\end{prop}

\section{Invariance properties of the bounded homological index}\label{s:invarpro}

\subsection{Homotopies in cyclic theory} 
Our approach to the invariance properties of our homological index is to use a variation on the well known notion from cyclic theory of homotopy invariance, see for example \cite[Lemma 1.21]{Kar:HCK} or \cite[Proposition 2.11]{CuQu:EPC}.
Thus we begin with $\C B$ a unital Banach algebra and we let $\C A$ be a unital $\cc$-algebra. 

Suppose that $\pi_t : \C A \to \C B$ is a unital algebra homomorphism for each $t \in [0,1]$ such that $t \mapsto \pi_t(x)$ is continuously differentiable for all $x \in \C A$. We will apply the notation $x_t := \pi_t(x)$ and $\frac{dx}{dt}$ for the derivative of $t \mapsto x_t$ noting that by assumption,  $\frac{dx}{dt} : [0,1] \to \C B$ is a continuous function.

For each $n \in \nn \cup \{0\}$, let $C_n^{\T{top}}(\C B) = \C B \hot \C B \hot \ldots \hot \C B$ denote the $(n+1)$-fold projective tensor product of $\C B$ with itself. Consider the linear map $J : C_{n+1}(\C A) \to C_n^{\T{top}}(\C B)$ defined by
\begin{equation}\label{eq:defjay}
J : C_{n+1}(\C A) \to C_n^{\T{top}}(\C B) \q J(x^0 \olo x^{n+1}) = \int_0^1 x^0_s \cd \frac{dx^1}{dt}\big|_s \ot x^2_s \olo x^{n+1}_s \, ds \\
\end{equation}

We now need two technical lemmas in preparation for our main result of this Section.
\begin{lemma}\label{l:homhoc} We have the relation
$
bJ + Jb = 0
$
\end{lemma}

\begin{proof}
Let $x = x^0 \olo x^{n+1} \in C_{n+1}(\C A)$. We have that
\[
\begin{split}
(Jd_i)(x) & = \int_0^1 x^0_s \cd \frac{dx^1}{dt}\big|_s \olo x^i_s x^{i+1}_s \olo x^{n+1}_s \, ds \\
& = (d_{i-1} J)(x)
\end{split}
\]
for all $i \in \{2,\ldots,n\}$. Likewise, we have that
\[
\begin{split}
(Jd_{n+1})(x) & = \int_0^1 x^{n+1}_s \cd x^0_s \cd \frac{dx^1}{dt}\big|_s \olo x^n_s \, ds \\
& = (d_n J)(x)
\end{split}
\]
Notice finally that
\[
\begin{split}
(Jd_0)(x) - (Jd_1)(x) 
& = \int_0^1 \big( x^0_s \cd x^1_s \cd \frac{dx^2}{dt}\big|_s - x^0_s \cd \frac{d(x^1 x^2)}{dt}\big|_s\big) \ot x^3_s \olo x^{n+1}_s \, ds\\
& = -\int_0^1 x^0_s \cd \frac{dx^1}{dt}\big|_s \cd x^2_s \ot x^3_s \olo x^{n+1}_s \, ds\\
& = -(d_0 J)(x)
\end{split}
\]
These computations imply that
\[
(Jb)(x) = \sum_{i=0}^{n+1} (-1)^i (Jd_i)(x) = -\sum_{i=0}^n (-1)^i (d_i J)(x) = -(bJ)(x)
\]
This proves the lemma.
\end{proof}

\begin{lemma}\label{l:homcon}
Let $x = x^0 \olo x^n \in C_n(\C A)$. Then
\[
(JB)(x) = \sum_{i=0}^n (-1)^{i \cd n}  \int_0^1 \frac{dx^i}{dt}\big|_s \olo x^n_s \ot x^0_s \olo x^{i-1}_s \, ds
\]
\end{lemma}
\begin{proof}
It is enough to show that 
\[
(J(1-t)s)(x) = \int_0^1 \frac{dx^0}{dt}\big|_s \ot x^1_s \olo x^n_s \, ds
\]
This follows since
\[
(Js)(x) = J(1 \ot x^0 \olo x^n) = \int_0^1 \frac{dx^0}{dt}\big|_s \ot x^1_s \olo x^n_s \, ds
\]
and since
\[
\begin{split}
-(Jts) & = -(-1)^{n+1} J(x^n \ot 1 \ot x^0 \olo x^{n-1}) \\
& = (-1)^n \int_0^1 x^n_s \cd \frac{d1}{dt}\big|_s \ot x^2_s \olo x^{n-1}_s \, ds
= 0.
\end{split}
\]
\end{proof}
The homological invariance result can now be stated.
\begin{thm}\label{t:hominv}
Let $x \in C_n^\la(\C A)$ and suppose that there exists an element $y \in C_{n+2}(\C A)$ such that $b(y) = B(x)$. Then the identity
$
\pi_0(x) - \pi_1(x) = (bJ)(y)
$
holds in $C_n^{\la,\T{top}}(\C B)$.
\end{thm}
\begin{proof}
For each $i \in \{0,\ldots,n\}$, let $\ka_i : C_n(\C A) \to C_n^{\T{top}}\big( C([0,1],\C B) \big)$ denote the linear map defined by
\[
\ka_i : z^0 \olo z^n \mapsto \pi(z^0) \olo \frac{d z^i}{dt} \olo \pi(z^n)
\]
where $\pi(z^j) : [0,1] \to \C B$ is given by $\pi(z^j)(t) := \pi_t(z^j) = z^j_t$. Furthermore, for each $s \in [0,1]$, let
\[
\T{ev}_s : C_n^{\T{top}}\big( C([0,1],\C B) \big) \to C_n^{\T{top}}(\C B)
\]
be the evaluation operator defined by $\T{ev}_s : \al_0 \olo \al_n \mapsto \al_0(s) \olo \al_n(s)$.
 
By Lemma \ref{l:homhoc} and Lemma \ref{l:homcon} we have that
\[
(bJ)(y) = -(Jb)(y) = -(JB)(x) = -\sum_{i=0}^n t^{n+1-i} \int_0^1 \T{ev}_s\big( \ka_i(x)\big) \, ds
\]
This implies that
\[
(bJ)(y) = -\sum_{i=0}^n \int_0^1 \T{ev}_s\big( \ka_i(x)\big) \, ds = \pi_0(x) - \pi_1(x)
\]
in the cyclic complex $C_n^{\la,\T{top}}(\C B) := C_n^{\T{top}}(\C B)/\T{Im}(1-t)$. But this is the desired identity.
\end{proof}

\subsection{Homotopy invariance properties}\label{ss:hominvpro}
Let $\C H$ be a Hilbert space and let $t \mapsto T_t^+$ be a continuously differentiable path of bounded operators (with respect to the operator norm).

For each $t \in [0,1]$, define the bounded operators
\[
\arr{ccc}{
T_t^- := (T_t^+)^* & R_t^- := 1 - T_t^+ \cd T_t^- & R_t^+ := 1 - T_t^- \cd T_t^+
}
\]

As usual, let $\C A$ denote the free unital algebra over $\cc$ with two generators $x,y$.

The assignments $\pi_t : x \mapsto T_t^+$ and $\pi_t : y \mapsto T_t^-$ define a unital algebra homomorphism $\pi_t : \C A \to \sL(\C H)$ for all $t \in [0,1]$. It follows by our assumption on the path $t \mapsto T_t^+$ that the map $t \mapsto \pi_t(z)$ is continuously differentiable in operator norm for all $z \in \C A$.

Let $n \in \nn$ and recall from Proposition \ref{p:twococ} that there exists a chain $\ga_2 \in C_2(\C I_n,\C A)$ such that $B(w^n - v^n) = -b(\ga_2)$. The chain $\ga_2 \in C_2(\C I_n,\C A)$ has the explicit form
\[
\begin{split}
\ga_2 & = -w^n \ot y \ot \om_n + v^n \ot \om_n \ot y - \om_n \ot w^n \ot y + \om_n \ot y \ot v^n \\ 
& \q - 2 \cd (w^n - v^n) \ot 1 \ot 1
\end{split}
\]
where $v := 1 - xy$, $w := 1- yx$ and $\om_n := x + vx \plp v^{n-1}x$.

It therefore follows immediately from Theorem \ref{t:hominv} that $\pi_0(w^n - v^n) - \pi_1(w^n - v^n)$ is a boundary in the cyclic complex $C_*^{\la,\T{top}}\big( \sL(\C H) \big)$. 

We shall now provide conditions on the path $t \mapsto T_t^+$ which entail that the difference $\pi_0(w^n - v^n) - \pi_1(w^n- v^n) = (R_0^+)^n - (R_0^-)^n - (R_1^+)^n + (R_1^-)^n$ becomes a boundary in the chain complex $C^{\la}_*(K)$ defined in Section \ref{s:tracoc}. This will imply an important invariance result for our bounded homological indices.

For each $t \in [0,1]$, we introduce the bounded selfadjoint operators
\[
R_t := \ma{cc}{R_t^+ & 0 \\ 0 & R_t^-} \q \T{and} \q T_t := \ma{cc}{0 & T_t^- \\ T_t^+ & 0}
\]
which act on the direct sum $\C H \op \C H$. Remark that the paths $t \mapsto T_t$ and $t \mapsto R_t$ are continuously differentiable in operator norm by our standing assumptions.

\begin{thm}\label{t:hominvII}
Let $n \in \nn$. Suppose that $(R_t^+)^n - (R_t^-)^n \in \sL^1(\C H)$ for all $t \in [0,1]$. Suppose furthermore that
$
\frac{d(R_t^n)}{dt} \, \, \T{ and } \, \,  \frac{d T_t}{dt} \cd R_t^n  \in \sL^1(\C H \op \C H)
$
for all $t \in [0,1]$ and that the maps
$
t \mapsto \frac{d(R_t^n)}{dt} \, \, \T{ and } \, \, t \mapsto \frac{d T_t}{dt} \cd R_t^n
$
are continuous in trace norm. 

Then $(R_0^+)^n - (R_0^-)^n - (R_1^+)^n + (R_1^-)^n \in \sL^1(\C H)$ determines the trivial element in $H^\la_0(K)$. In particular, we obtain that the homological indices in degree $n$ of $T_0^+$ and $T_1^+$ coincide, $\T{H-Ind}_n(T_0^+) = \T{H-Ind}_n(T_1^+)$.
\end{thm}
\begin{proof}
Recall from \eqref{eq:defjay} that the chain $J(\ga_2) \in C_1^{\T{top}}(\sL(\C H))$ is given explicitly by the formula
\begin{equation}\label{eq:jayexp}
\begin{split}
J(\ga_2) & = -\int_0^1 (R_t^+)^n \cd \frac{d T_t^-}{dt} \ot \Om_t \, dt
+ \int_0^1 (R_t^-)^n \cd \frac{d \Om_t}{dt} \ot T_t^- \, dt \\
& \q - \int_0^1 \Om_t \cd \frac{d((R_t^+)^n)}{dt} \ot T_t^- \, dt
+ \int_0^1 \Om_t \cd \frac{d(T_t^-)}{dt} \ot (R_t^-)^n \, dt
\end{split}
\end{equation}
where $\Om_t := T_t^+ + R_t^- T_t^+ \plp (R_t^-)^{n-1} T_t^+ = T_t^+ + T_t^+ R_t^+ \plp T_t^+ (R_t^+)^{n-1}$ for all $t \in [0,1]$.

By Theorem \ref{t:hominv} it suffices to show that $J(\ga_2)$ determines an element in the Banach space $K_1$ defined in Section \ref{s:tracoc}. Thus, we need to show that each of the integrands in \eqref{eq:jayexp} define continuous maps $[0,1] \to \C K_1$. By the definition of the norm $\|\cd \|_1 : \C K_1 \to [0,\infty)$ it suffices to show that the following eight maps
\[
\begin{split}
t\mapsto  (R_t^+)^n \cd \frac{d T_t^-}{dt} \cd \Om_t & \qq t\mapsto  \Om_t \cd (R_t^+)^n \cd \frac{d T_t^-}{dt} \\
t \mapsto (R_t^-)^n \cd \frac{d \Om_t}{dt} \cd T_t^- & \qq t\mapsto T_t^- \cd (R_t^-)^n \cd \frac{d \Om_t}{dt} \\
t \mapsto \Om_t \cd \frac{d\big((R_t^+)^n\big)}{dt} \cd T_t^- & \qq t \mapsto T_t^- \cd \Om_t \cd \frac{d\big((R_t^+)^n\big)}{dt} \\
t \mapsto \Om_t \cd \frac{d(T_t^-)}{dt} \cd (R_t^-)^n & \qq t \mapsto (R_t^-)^n \cd \Om_t \cd \frac{d(T_t^-)}{dt}
\end{split}
\]
are continuous in trace norm, see Definition \ref{d:tranorcha}. This is easily seen to be implied by the assumptions of our theorem for path $1$, $2$, $5$, $6$ and $7$. To see that path $8$ is continuous in trace norm we simply remark that $(R_t^-)^n \cd \Om_t = \Om_t \cd (R_t^+)^n$. To see that path $3$ and $4$ are continuous in trace norm it is enough to show that the path $t \mapsto (R_t^-)^n \cd \frac{d \Om_t}{dt}$ is continuous in trace norm.

To this end, we remark that
\[
\begin{split}
(R_t^-)^n \cd \frac{d \Om_t}{dt} & = \big(1 + R_t^- \plp (R_t^-)^{n-1} \big)\cd (R_t^-)^n \cd \frac{d(T_t^+)}{dt} \\
& \qq + (R_t^-)^n \cd \Big( \frac{d(R_t^-)}{dt}  \plp \frac{d\big( (R_t^-)^{n-1}\big)}{dt}\Big) \cd T_t^+
\end{split}
\]
for all $t \in [0,1]$. It therefore suffices to show that the path $t \mapsto (R_t^-)^n \cd \frac{d(R_t^-)}{dt}$ is continuous in trace norm. But this follows since
\[
\begin{split}
(R_t^-)^n \cd \frac{d(R_t^-)}{dt} & = - (R_t^-)^n \cd \frac{d(T_t^+)}{dt} \cd T_t^- -
(R_t^-)^n \cd T_t^+ \cd \frac{d(T_t^-)}{dt} \\
& = - (R_t^-)^n \cd \frac{d(T_t^+)}{dt} \cd T_t^- -
T_t^+ \cd (R_t^+)^n \cd \frac{d(T_t^-)}{dt}
\end{split}
\]
for all $t \in [0,1]$.
\end{proof}

\section{Homological indices of unbounded operators}\label{s:homindunb}
The interesting examples of the homological index arise from differential operators. In this Section we formulate
a general framework that accommodates the unbounded operators for which we will be able to apply our invariance arguments of the previous Section.
Consider a closed densely defined operator $\C D^+ : \T{Dom}(\C D^+) \to \C H$ on a Hilbert space $\C H$. Let $\C D^- := (\C D^+)^* : \T{Dom}(\C D^-) \to \C H$ denote the adjoint of $\C D^+$. {\color{black} We now introduce an unbounded version of Definition \ref{d:homindbou}.

\begin{dfn}\label{d:homindunb}
Let $n \in \nn$. We will say that the \emph{homological index of $\C D^+$ exists in degree $n$} when the bounded operator
\[
(1 + \C D^- \C D^+)^{-n} - (1 + \C D^+ \C D^-)^{-n} : \C H \to \C H
\]
is of trace class. In this case, the \emph{homological index of $\C D^+$ in degree $n$} is defined by
\[
\T{H-Ind}_n(\C D^+) := \T{Tr}\Big( (1 + \C D^- \C D^+)^{-n} - (1 + \C D^+ \C D^-)^{-n} \Big)
\]
\end{dfn}

The link between the unbounded and the bounded version of the homological index can be explained as follows.
%
%
Define the bounded operator $$T^+ := \C D^+(1 + \C D^- \C D^+)^{-1/2} : \C H \to \C H$$ and notice that
\[
1 - T^+(T^+)^* = (1 + \C D^+ \C D^-)^{-1} \q \T{and} \q 1 - (T^+)^* T^+ = (1 + \C D^-\C D^+)^{-1}
\]
It therefore follows that the homological index of $T^+$ in degree $n$ exists if and only if the homological index of $\C D^+$ exists in degree $n$. Furthermore, we have the identities
\[
\begin{split}
\T{H-Ind}_n(\C D^+) & = \T{Tr}\Big( (1 + \C D^- \C D^+)^{-n} - (1 + \C D^+ \C D^-)^{-n} \Big) \\
& = \T{Tr}\big( (1 - (T^+)^*T^+)^n - (1 - T^+ (T^+)^*)^n \big) = \T{H-Ind}_n(T^+)
\end{split}
\]
}
We are interested in studying the invariance of the homological index of $\C D^+$ under a relevant class of unbounded perturbations. More precisely, the main aim of the next two Sections is to find a good class of closed unbounded operators $B^+ : \T{Dom}(B^+) \to \C H$ such that the homological index of $\C D^+ + B^+$ exists \emph{and} coincides with the homological index of $\C D^+$.

We have collected the results which we need about perturbations of unbounded operators in an Appendix to this paper.


\section{Existence of the unbounded homological index}\label{s:exiunbhom}
Let $\C D^+ : \T{Dom}(\C D^+) \to \C H$ and $B^+ : \T{Dom}(B^+) \to \C H$ be closed unbounded operators. The Hilbert space adjoints of $\C D^+$ and of $B^+$ are denoted by $\C D^- := (\C D^+)^* : \T{Dom}\big((\C D^+)^* \big) \to \C H$ and $B^- := (\C B^+)^* : \T{Dom}\big((B^+)^* \big) \to \C H$.

We will apply the notation
\[
\begin{split}
\C D & := \ma{cc}{0 & \C D^- \\ \C D^+ & 0} : \T{Dom}(\C D^+) \op \T{Dom}(\C D^-) \to \C H \op \C H \q \T{and} \\
B & := \ma{cc}{0 & B^- \\ B^+ & 0} : \T{Dom}(B^+) \op \T{Dom}(B^-) \to \C H \op \C H
\end{split}
\]
for the selfadjoint unbounded operators associated with $\C D^+$ and $B^+$.
 
Suppose now that $\T{Dom}(\C D^+) \su \T{Dom}(B^+)$ and that $\T{Dom}(\C D^-) \su \T{Dom}(B^-)$. It then follows by Lemma \ref{l:clopat} that the unbounded operator $\C D^+ + t \cd B^+ : \T{Dom}(\C D^+) \to \C H$ is closable for all $t \in [0,1]$. The closures will be denoted by 
\[
\C D^+_t := \ov{\C D^+ + t\cd B^+} : \T{Dom}(\C D^+_t) \to \C H
\]
We will apply the notation $\C D^-_t := (\C D^+_t)^* : \T{Dom}\big((\C D^+_t)^*\big) \to \C H$ for the adjoints.

Introduce the notation
\[
\begin{split}
& \C D_t := \ma{cc}{0 & \C D_t^- \\ \C D_t^+ & 0} : \T{Dom}(\C D_t^+) \op \T{Dom}(\C D_t^-) \to \C H \op \C H \q \T{and} \\
& \De_t := \ma{cc}{\De_t^+ & 0 \\ 0 & \De_t^-} := \C D_t^2 = \ma{cc}{\C D^-_t \C D^+_t & 0 \\ 0 & \C D^+_t \C D^-_t}
\end{split}
\]
for all $t \in [0,1]$. We remark that $\C D_t$ is selfadjoint and that $\De_t$ is positive and selfadjoint for all $t \in [0,1]$.

The \emph{bounded transform} of $\C D^+_t$ is then defined as the bounded operator $T^+_t := \C D^+_t \big(1 + \De^+_t\big)^{-1/2}$. The adjoint of $T^+_t : \C H \to \C H$ agrees with the bounded transform of $\C D^-_t$, thus $T^-_t := (T^+_t)^* = \C D^-_t \big(1 + \De^-_t \big)^{-1/2}$. The \emph{resolvent} of $\De_t$ is the bounded operator defined by $R_t := (1 + \De_t)^{-1} : \C H \op \C H \to \C H \op \C H$.
\bigskip

\emph{The following Assumption will remain in effect throughout this Section:}

\begin{assu}\label{a:patrelmai}
Suppose that the following holds:
\begin{enumerate}
\item $\T{Dom}(\C D^+) \su \T{Dom}(B^+)$ and $\T{Dom}(\C D^-) \su \T{Dom}(B^-)$.
\item There exists a dense subspace $\C E \su \C H \op \C H$ such that $(i + \C D_t)^{-1}(\xi) \in \T{Dom}(B)$ for all $\xi \in \C E$ and all $t \in [0,1]$.
\item The unbounded operator $B \cd (i + \C D_t)^{-1} : \C E \to \C H \op \C H$ extends to a bounded operator $X_t : \C H \op \C H \to \C H \op \C H$ for all $t \in [0,1]$.
\item There exists a $K > 0$ such that $\| X_t \| \leq K$ for all $t \in [0,1]$.
\end{enumerate}
\end{assu}

We remark that the conditions in the above Assumption summarizes the conditions in Assumption \ref{a:patrelbouI} and Assumption \ref{a:patrelbouII} in the Appendix. In particular, it follows by Proposition \ref{p:dompat} that $\T{Dom}(\C D_t) = \T{Dom}(\C D)$ for all $t \in [0,1]$.
\bigskip

Our main interest is now to find relevant conditions on the selfadjoint unbounded operator $B$ such that the homological index for $\C D^+ : \T{Dom}(\C D^+) \to \C H$ exists if and only if the homological index for $\C D^+ + B^+ : \T{Dom}(\C D^+) \to \C H$ exists. This will be carried out by studying the difference of powers of resolvents $R_t^n - R_s^n$ for $n \in \nn$ and $t,s \in [0,1]$. To reach our main result of this section, we need to prove a few preliminary Lemmas. Our first Lemma relies on the resolvent identity: 

\begin{lemma}\label{l:resexpa}
Let $N \in \nn$ be a positive integer such that $N > \sup_{t \in [0,1]} \| X_t \|$. Let $t,s \in [0,1]$ with $|t-s| \leq 1/N$. We then have the identity
\[
R_t = \big(1 + (t-s) X_s^* \big)^{-1} \cd R_s \cd \big(1 + (t - s) X_s\big)^{-1}
\]
\end{lemma}
\begin{proof}
By the resolvent identity in Lemma \ref{l:resideII}, we have that
\[
(i + \C D_t)^{-1} - (i + \C D_s)^{-1} =  - (i + \C D_t)^{-1} \cd (t-s) \cd X_s
\]
This is equivalent to the identity
\[
(i + \C D_s)^{-1} = (i + \C D_t)^{-1} \cd \big( 1 + (t-s) \cd X_s \big)
\]
Now, since $|t -s|\cd \|X_s \| < 1$ we obtain that
\[
(i + \C D_t)^{-1} = (i + \C D_s)^{-1} \cd \big(1 + (t -s)\cd X_s \big)^{-1}
\]
This implies that 
\[
R_t = (-i+ \C D_t)^{-1} (i + \C D_t)^{-1} =  \big(1 + (t -s)\cd X_s^* \big)^{-1} \cd R_s \cd \big(1 + (t -s)\cd X_s \big)^{-1}
\]
which is the desired identity.
\end{proof}

The identity obtained in the last Lemma plays a central role in the proof of the next Lemma which lies at the technical core of the present Section.

\begin{lemma}\label{l:tecsum}
Let $N \in \nn$ be a positive integer with {\color{black} $N > \sup_{t \in [0,1]} \| X_t \|$, let {\color{black} $s \in [0,1]$} and let $n \in \nn$.} Suppose that
$
R_s^j \cd  X_s \cd R_s^k \in \sL^{n/(j+k)}(\C H \op \C H)
$
for all $j,k \in \{0,\ldots,n\}$ with $1 \leq j + k \leq n$. Then the assignment
$
t \mapsto R_s^j \cd (R_t^m - R_s^m) \cd R_s^k
$
determines a continuous map {\color{black} $[s -1/N,s+ 1/N] \cap [0,1] \to \sL^{n/(m+ j +k)}(\C H \op \C H)$} for all $m \in \{1,\ldots,n\}$ and all $j,k \in \{0,\ldots,n-m\}$ with $0 \leq j+k \leq n-m$.
\end{lemma}
\begin{proof}
To ease the notation, let $Y_t := \big( 1 + (t-s) X_s\big)^{-1}$ for all {\color{black} $t \in [s -1/N,s+ 1/N] \cap [0,1]$. Remark that $t \mapsto Y_t$ defines a continuous path $[s - 1/N, s + 1/N] \cap [0,1] \to \sL(\C H \op \C H)$. Notice furthermore that}
\[
Y_t - 1 = (s - t) \cd X_s \cd Y_t = (s - t) \cd Y_t \cd X_s
\]
The proof now runs by induction on $m \in \{1,\ldots,n\}$.

To begin the induction we use Lemma \ref{l:resexpa} to obtain
\begin{equation}\label{eq:resadv}
\begin{split}
R_t - R_s 
& = Y_t^* \cd R_s \cd Y_t - R_s \\
& = ( Y_t^* - 1) \cd R_s \cd Y_t + R_s \cd (Y_t - 1) \\
& = (s - t) \cd \big( Y_t^* \cd X_s^* \cd R_s \cd Y_t + R_s \cd X_s \cd Y_t \big)
\end{split}
\end{equation}
for all $t \in [s -1/N,s+ 1/N] \cap [0,1]$.
Let now $j,k \in \{0,\ldots,n-1\}$ with $0 \leq j+k \leq n-1$ and compute as follows,
\begin{equation}\label{eq:eascomhol}
\begin{split}
R_s^{j+1} \cd X_s \cd Y_t \cd R_s^k 
& = R_s^{j+1} \cd X_s \cd R_s^k + R_s^{j+1} \cd X_s \cd (Y_t - 1) \cd R_s^k \\
& = R_s^{j+1} \cd X_s \cd R_s^k + (s - t) \cd R_s^{j+1} \cd X_s \cd Y_t \cd X_s \cd R_s^k
\end{split}
\end{equation}
An application of the H\"older inequality then implies that
$
t \mapsto R_s^{j+1} \cd X_s \cd Y_t \cd R_s^k
$
determines a continuous map $[s - 1/N, s+ 1/N] \cap [0,1] \to \sL^{n/(j+1+k)}(\C H \op \C H)$.
Next, a computation similar to the one given in \eqref{eq:eascomhol} shows that
\[
t \mapsto R_s^j \cd Y_t^* \cd X_s^* \cd R_s \cd Y_t \cd R_s^k
\]
also determines a continuous map $[s - 1/N, s+ 1/N] \cap [0,1] \to \sL^{n/(j+1+k)}(\C H \op \C H)$.

The identity in \eqref{eq:resadv} now entails that
$t \mapsto R_s^j \cd (R_t - R_s) \cd R_s^k$ determines a continuous map $[s - 1/N, s+ 1/N] \cap [0,1] \to \sL^{n/(j+1+k)}(\C H \op \C H)$ for all $j,k \in \{0,\ldots,n-1\}$ with $0 \leq j+k \leq n-1$. This establishes the first step in the induction argument.

Let now $m \in \{1,\ldots,n-1\}$ and suppose that
$
t \mapsto R_s^j \cd (R_t^m - R_s^m) \cd R_s^k
$
determines a continuous map $[s - 1/N , s+ 1/N] \cap [0,1] \to \sL^{n/(m+j+k)}(\C H \op \C H)$ for all $j,k \in \{0,\ldots,n-m\}$ with $0 \leq j+k \leq n-m$.
Consider $j,k \in \{0,\ldots,n-m - 1\}$ with $0 \leq j+k \leq n - m - 1$. We have that
\[
\begin{split}
& R_s^j \cd (R_t^{m+1} - R_s^{m+1}) \cd R_s^k \\
& \q = R_s^j \cd (R_t - R_s) \cd R_t^m \cd R_s^k
+ R_s^{j+1} \cd (R_t^m - R_s^m) \cd R_s^k \\
& \q = R_s^j \cd (R_t - R_s) \cd (R_t^m - R_s^m) \cd R_s^k
+ R_s^j \cd (R_t - R_s) \cd R_s^{m+k} \\
& \qq + R_s^{j+1} \cd (R_t^m - R_s^m) \cd R_s^k
\end{split}
\]
for all $t \in [s - 1/N , s+ 1/N] \cap [0,1]$. An application of the H\"older inequality now implies that
$t \mapsto R_s^j \cd (R_t^{m+1} - R_s^{m+1}) \cd R_s^k$ determines a continuous map from $[s - 1/N , s+ 1/N] \cap [0,1] \to \sL^{n/(m + 1 + j + k)}(\C H \op \C H)$. This proves the induction step.
\end{proof}

In order to make full use of the above Lemma, we need to analyze the conditions a little more. This is carried out in the next Lemma.

\begin{lemma}\label{l:contecsum}
Let $n \in \nn$. The following two statements are equivalent:
\begin{enumerate}
\item There exists an $s_0 \in [0,1]$ such that $R_{s_0}^j \cd  X_{s_0} \cd R_{s_0}^k \in \sL^{n/(j+k)}(\C H \op \C H)$ for all $j,k \in \{0,\ldots,n\}$ with $1 \leq j + k \leq n$.
\item $R_s^j \cd  X_s \cd R_s^k \in \sL^{n/(j+k)}(\C H \op \C H)$ for all $j,k \in \{0,\ldots,n\}$ with $1 \leq j + k \leq n$ and all $s \in [0,1]$.
\end{enumerate}
\end{lemma}
\begin{proof}
We will only prove that $(1)$ implies $(2)$, the reverse argument being trivial.
Let therefore $s_0 \in [0,1]$ and suppose that $R_{s_0}^j \cd  X_{s_0} \cd R_{s_0}^k \in \sL^{n/(j+k)}(\C H \op \C H)$ for all $j,k \in \{0,\ldots,n\}$ with $1 \leq j + k \leq n$.

Choose a positive integer $N \in \nn$ such that
$
N > \sup_{t \in [0,1]} \| X_t \|.
$
It is then enough to show that $R_s^j \cd  X_s \cd R_s^k \in \sL^{n/(j+k)}(\C H \op \C H)$ for all $j,k \in \{0,\ldots,n\}$ with $1 \leq j + k \leq n$ and all $s \in [s_0 - 1/N,s_0 + 1/N] \cap [0,1]$.
Let therefore $j,k \in \{0,\ldots,n\}$ with $1 \leq j + k \leq n$ and $s \in [s_0 - 1/N,s_0 + 1/N] \cap [0,1]$ be fixed.

Repeated use of the resolvent identity (see Lemma \ref{l:resideIII}) yields that
\begin{equation}\label{eq:resreside}
\begin{split}
X_s & = X_{s_0} + X_s - X_{s_0} \\
& = X_{s_0} + X_s \cd (s_0 - s) \cd X_{s_0} \\
& = X_{s_0} + X_{s_0} \cd (s_0 - s) \cd X_{s_0} + (X_s - X_{s_0}) \cd (s_0 - s) \cd X_{s_0} \\
& = X_{s_0} + X_{s_0} \cd (s_0 - s) \cd X_{s_0} + X_{s_0} \cd (s_0 - s) \cd X_s \cd (s_0 - s) \cd X_{s_0}
\end{split}
\end{equation}
It then follows by Lemma \ref{l:tecsum} and the H\"older inequality that $R_s^j \cd  X_s \cd R_s^k \in \sL^{n/(j+k)}(\C H \op \C H)$.
\end{proof}

We are now ready to prove the main result of this Section. It provides certain summability conditions on the {\color{black} selfadjoint unbounded} operator $B = \ma{cc}{0 & B^- \\ B^+ & 0}$. These conditions entail that the homological index of $\C D^+ + B^+$ exists whenever the homological index of $\C D^+$ exists.
%

\begin{thm}\label{t:exihomind}
Suppose that the closed unbounded operators $\C D^+ : \T{Dom}(\C D^+) \to \C H$ and $B^+ : \T{Dom}(B^+) \to \C H$ satisfy the conditions of Assumption \ref{a:patrelmai}. Let $n \in \nn$ and suppose furthermore that
$R_0^j \cd  X_0 \cd R_0^k \in \sL^{n/(j+k)}(\C H \op \C H)$ for all $j,k \in \{0,\ldots,n\}$ with $1 \leq j + k \leq n$. Then the homological index $\T{H-Ind}_n(\C D^+ + B^+)$ exists if and only if the homological index of $\T{H-Ind}_n(\C D^+)$ exists.
\end{thm}

The above theorem follows immediately from the next (stronger) Proposition. Indeed, this proposition implies that $R_1^n - R_0^n : \C H \op \C H \to \C H \op \C H$ is of trace class. But this means that both of the operators
\[
(1 + \De_1^+)^{-n} - (1 + \De_0^+)^{-n} : \C H \to \C H \q \T{and} \q
(1 + \De_1^-)^{-n} - (1 + \De_0^-)^{-n} : \C H \to \C H
\]
are of trace class. And therefore we must have that
\[
\Big( (1 + \De_1^+)^{-n} - (1 + \De_1^-)^{-n} \in \sL^1(\C H) \Big) 
\lrar \Big( (1 + \De_0^+)^{-n} - (1 + \De_0^-)^{-n} \in \sL^1(\C H) \Big)
\]

\begin{prop}\label{p:resdiflam}
Let $n \in \nn$ and suppose that
\[
R_0^j \cd  X_0 \cd R_0^k \in \sL^{n/(j+k)}(\C H \op \C H)
\]
for all $j,k \in \{0,\ldots,n\}$ with $1 \leq j + k \leq n$. 
Then the assignment
$
t \mapsto R_t^m - R_0^m
$
determines a continuous map $[0,1] \to \sL^{n/m}(\C H \op \C H)$ for all $m \in \{1,\ldots,n\}$. 
%
\end{prop}
\begin{proof}
This is an easy consequence of Lemma \ref{l:contecsum} and Lemma \ref{l:tecsum}.
\end{proof}

\section{Invariance properties of the unbounded homological index}\label{s:invunbhom}

In this Section we will prove our main result, Theorem \ref{t:hominvunb} about invariance of the homological index under perturbations of our unbounded operators. The notation which we introduced in Section \ref{s:exiunbhom} will be applied throughout this Section. In particular, $\C D^+ : \T{Dom}(\C D^+) \to \C H$ and $B^+ : \T{Dom}(B^+) \to \C H$ are closed unbounded operators.
\bigskip

\emph{It will be assumed throughout this Section that $\C D^+$ and $B^+$ satisfy the conditions described in Assumption \ref{a:patrelmai}.}
\bigskip

%

{\color{black} Recall from the discussion in Section \ref{s:homindunb} that the homological index of $\C D^+$ exists if and only if the homological index of the bounded transform $T^+ := \C D^+ (1 + \C D^- \C D^+)^{-1/2}$ exists. And in this case, we have the identity $\T{H-Ind}_n(\C D^+) = \T{H-Ind}_n(T^+)$. For our purposes it is therefore enough to compare the homological indices of $T_0^+$ and $T_1^+$. And this can be done using Theorem \ref{t:hominvII}.}

{\color{black} As in Subsection \ref{ss:hominvpro} we apply the notation
\[
T_t := \ma{cc}{0 & (T_t^+)^* \\
T_t^+ & 0
} 
\q \T{and} \q
R_t := 1 - T_t^2 = \ma{cc}{ 1 - T_t^- T_t^+ & 0 \\ 0 & 1 - T_t^+ T_t^-}
\]
where $T_t := \C D_t \cd (1 + \C D_t^2)^{-1/2}$ denotes the bounded transform of $\C D_t := \C D + t \cd B : \T{Dom}(\C D) \to \C H \op \C H$. Our strategy is thus to prove that the path of bounded operators $t \mapsto T_t$ satisfy the conditions of Theorem \ref{t:hominvunb}.}

Since the technicalities involved in showing that the path $t \mapsto T_t$ is continuously differentiable are substantial (see for example \cite[Section 6]{CaPoSu:SIF}), we will avoid this question in the current paper and simply make the following standing assumption:

\begin{assu}\label{a:patrelbou}
Suppose that the closed unbounded operators $\C D^+ : \T{Dom}(\C D^+) \to \C H$ and $B^+ : \T{Dom}(B^+) \to \C H$ satisfy the conditions of Assumption \ref{a:patrelmai}. Suppose furthermore that the path of bounded transforms $t \mapsto T_t$ is continuously differentiable in operator norm.
\end{assu}

\emph{Unless the contrary is explicitly mentioned, the above Assumption will remain in effect throughout this Section.}
\bigskip

%

{\color{black} After this preliminary discussion we prove the first half of the conditions in Theorem \ref{t:hominvII}.}

\begin{prop}\label{p:traconder}
Let $n \in \nn$ and suppose that
$
(R_0)^j \cd X_0 \cd (R_0)^k \in \sL^{n/(j+k)}(\C H \op \C H)
$
for all $j,k \in \{0,\ldots,n\}$ with $1 \leq j + k \leq n$. Then the derivative $\frac{d \big((R_t)^n\big)}{dt}$ defines a continuous map
$
\frac{d \big((R_t)^n\big)}{dt} : [0,1] \to \sL^1(\C H \op \C H).
$
\end{prop}
\begin{proof}
By Proposition \ref{p:derpowres} we have that
\[
\frac{d \big((R_t)^n\big)}{dt}\big|_{t_0} = - \sum_{j=0}^{n-1} (R_{t_0})^j \cd \big( R_{t_0} \cd X_{t_0} + X_{t_0}^* \cd R_{t_0} \big) \cd (R_{t_0})^{n-1 - j}
\]
for all $t_0 \in [0,1]$. Since the adjoint operation determines an isometry $* : \sL^1(\C H \op \C H) \to \sL^1(\C H \op \C H)$ it is therefore enough to show that
\[
t \mapsto (R_t)^j \cd X_t \cd (R_t)^{n - j}
\]
defines a continuous map $[0,1] \to \sL^1(\C H \op \C H)$ for all $j \in \{0,\ldots,n\}$. But this is a consequence of Proposition \ref{p:resdiflam}, the identity in \eqref{eq:resreside} and the H\"older inequality.
\end{proof}

{\color{black} Our next aim is to prove the trace norm continuity of the path $t \mapsto \frac{d T_t}{dt} \cd (R_t)^n$ under the following summability conditions:

\begin{assu}\label{a:sumpat}
Let $n \in \nn$ and suppose that
$
(R_0)^j \cd X_0 \cd (R_0)^k \in \sL^{n/(j+k)}(\C H \op \C H)
$
for all $j,k \in \{0,\ldots,n\}$ with $1 \leq j + k \leq n$. Suppose furthermore that there exists an $\ep \in (0,1/2)$ such that
$
X_0 \cd (R_0)^{n- \ep} \in \sL^1(\C H \op \C H).
$
\end{assu}
\bigskip

\emph{The above Assumption will remain in effect for the rest of this Section.}}
\bigskip

{\color{black} Remark now that it follows from Proposition \ref{p:resdiflam} that the path $t \mapsto \frac{d T_t}{dt}\big|_{t_0} \cd (R_t)^n$ is continuous in trace norm if and only if the path $t_0 \mapsto \frac{d T_t}{dt}\big|_{t_0} \cd (R_0)^n$ is continuous in trace norm. Furthermore, it follows from Proposition \ref{p:derboutra} that we have the explicit formula:
\begin{equation}\label{eq:traconsum}
\begin{split}
\frac{d T_t}{dt}\Big|_{t_0} \cd (R_0)^n 
&
= (1 + \C D_{t_0}^2)^{-1/2} \cd X_0 \cd (R_0)^{n-1} \cd (-i + \C D)^{-1} \\
& \q - \frac{1}{\pi} \cd \int_0^\infty \mu^{-1/2} \cd 
\big( R_{t_0}^{1 + \mu} \cd X_{t_0}^{1 + \mu} + (X_{t_0}^{1 + \mu})^* \cd R_{t_0}^{1 + \mu} \big)
\, d\mu \cd \C D_{t_0} \cd (R_0)^n
\end{split}
\end{equation}
for all $t_0 \in [0,1]$. 
In order to show that the path $t \mapsto \frac{d T_t}{dt}\big|_{t_0} \cd (R_t)^n$ is continuous in trace norm it is therefore enough to show that each of the paths
\begin{equation}\label{eq:twotracon}
\begin{split}
& t \mapsto (1 + \C D_t^2)^{-1/2} \cd X_0 \cd (R_0)^{n-1} \cd (-i + \C D)^{-1} \q \T{and} \\
& t \mapsto \int_0^\infty \mu^{-1/2} \cd 
\big( R_t^{1 + \mu} \cd X_t^{1 + \mu} + (X_t^{1 + \mu})^* \cd R_t^{1 + \mu} \big)
\, d\mu \cd \C D_t \cd (R_0)^n
\end{split}
\end{equation}
are continuous in trace norm.
}

We begin with the first one:

\begin{lemma}\label{l:firpat}
The path
$
t \mapsto (1 + \C D_t^2)^{-1/2} \cd X_0 \cd (R_0)^{n-1} \cd (-i + \C D)^{-1}
$
is continuous in trace norm.
\end{lemma}
\begin{proof}
For each $t \in [0,1]$ we have that
\[
(1 + \C D_t^2)^{-1/2} \cd X_0 \cd (R_0)^{n-1} \cd (-i + \C D)^{-1} 
= (-i + \C D_t) \cd (1 + \C D_t^2)^{-1/2} \cd (X_t)^* \cd (R_0)^n
\]
By the H\"older inequality it is therefore enough to show that
\[
t \mapsto (-i + \C D_t) \cd (1 + \C D_t^2)^{-1/2} = -i \cd (1 + \C D_t^2)^{-1/2} + T_t
\]
is continuous in operator norm and that
\[
t \mapsto (X_t)^* \cd (R_0)^n = (X_0)^* \cd (R_0)^n - t \cd \big( X_0 \cd X_t)^* \cd (R_0)^n
\]
is continuous in trace norm. But this is a consequence of our general conditions (Assumption \ref{a:patrelbou} and Assumption \ref{a:sumpat}) and Lemma \ref{l:dersqu}.
\end{proof}

In order to show that the second path in \eqref{eq:twotracon} is continuous in trace norm we first analyze the integrand:

\begin{lemma}\label{l:traconint}
Both of the paths
\begin{equation}\label{eq:twopattra}
t \mapsto R_t^\la \cd X_t^\la \cd \C D_t \cd (R_0)^n \q \T{and} \q
t \mapsto (X_t^\la)^* \cd R_t^\la \cd \C D_t \cd (R_0)^n
\end{equation}
are continuous in trace norm for all $\la > 0$. Furthermore, there exists a constant $K > 0$ such that
\[
\begin{split}
\big\| R_t^\la \cd X_t^\la \cd \C D_t \cd (R_0)^n \big\|_1 \, \, \, , \, \, \,
\big\| (X_t^\la)^* \cd R_t^\la \cd \C D_t \cd (R_0)^n \big\|_1 \leq K \cd \la^{-1/2 - \ep}
\end{split}
\]
for all $t \in [0,1]$ and all $\la \geq 1$.
\end{lemma}
\begin{proof}
We start by computing as follows
\[
\begin{split}
R_t^\la \cd X_t^\la \cd \C D_t \cd (R_0)^n  &
 = R_t^\la \cd B \cd (R_0)^n - i \cd \la^{1/2} \cd R_t^\la \cd X_t^\la \cd (R_0)^n \\
& 
= (i \cd \la^{1/2} + \C D_t)^{-1} \cd (X_t^\la)^* \cd (R_0)^n
- i \cd \la^{1/2} \cd R_t^\la \cd X_t^\la \cd (R_0)^n
\end{split}
\]
for all $t \in [0,1]$ and all $\la > 0$.
This computation proves that the first path in \eqref{eq:twopattra} is continuous in trace norm. Furthermore, there exists a constant $K_1 > 0$ such that
\[
\big\| R_t^\la \cd X_t^\la \cd \C D_t \cd (R_0)^n \big\|_1
\leq K_1 \cd \la^{-1/2} \cd \Big( \big\| (R_0)^n \cd X_t^\la \big\|_1 +  \big\| X_t^\la \cd (R_0)^n \big\|_1 \Big)
\]
for all $t \in [0,1]$ and all $\la > 0$. The desired estimate on the first path in \eqref{eq:twopattra} therefore follows from the next Lemma.
A similar argument proves the required statements for the second path in \eqref{eq:twopattra}.
\end{proof}
%

\begin{lemma}\label{l:tecinequa}
There exists a constant $K > 0$ such that 
\[
\big\| X_t^\la \cd (R_0)^n \big\|_1 \, \, \, , \, \, \, \big\| (R_0)^n \cd X_t^\la \big\|_1 \leq K \cd \la^{-\ep}
\]
for all $t \in [0,1]$ and all $\la \geq 1$.
\end{lemma}
\begin{proof}
We start by computing as follows,
\[
\begin{split}
X_t^\la \cd (R_0)^n 
& = (1 - tX_t^\la) \cd X_0^\la \cd (R_0)^n \\
& = (1 - tX_t^\la) \cd X_0 \cd (i + \C D_0) \cd (i \la^{1/2} + \C D_0)^{-1} \cd (R_0)^n \\
& = (1 - tX_t^\la) \cd X_0 \cd (R_0)^{n-\ep} \cd (i + \C D_0) \cd (i \la^{1/2} + \C D_0)^{-1} \cd (R_0)^{\ep}
\end{split}
\]
for all $t \in [0,1]$ and $\la > 0$.
We then remark that there exists a constant $K_1 > 0$ such that
\[
\| X_t^\la \| \leq K_1 \q \T{and} \q \big\| (i + \C D_0) \cd (i \la^{1/2} + \C D_0)^{-1} \cd (R_0)^{\ep} \big\| \leq K_1 \cd \la^{-\ep}
\]
for all $\la \geq 1$ and all $t \in [0,1]$.
Combining these observations, we may choose a constant $K_2 > 0$ such that
$
\big\| X_t^\la \cd (R_0)^n \big\|_1 \leq K_2 \cd \la^{-\ep}
$
for all $t \in [0,1]$ and all $\la \geq 1$.
A similar argument proves the required trace norm estimate on the bounded operators $(R_0)^n \cd X_t^\la$.
\end{proof}

\begin{prop}\label{p:consqu}
The assignment
\[
t_0 \mapsto \frac{d T_t}{dt}\Big|_{t_0} \cd (R_{t_0})^n
\]
defines a continuous path $[0,1] \to \sL^1(\C H \op \C H)$.
\end{prop}
\begin{proof}
By Lemma \ref{l:firpat} and Equation \eqref{eq:traconsum} it is enough to show that the path
\[
t \mapsto \int_0^\infty \mu^{-1/2} \cd 
\big( R_t^{1 + \mu} \cd X_t^{1 + \mu} + (X_t^{1 + \mu})^* \cd R_t^{1 + \mu} \big) \cd \C D_t \cd (R_0)^n \, d\mu 
\]
is continuous in trace norm.
Therefore, consider the algebra $C\big( [0,1], \sL^1(\C H \op \C H) \big)$ of continuous paths in the Banach algebra $\sL^1(\C H \op \C H)$. This algebra becomes a Banach algebra when equipped with the norm $\| \cd \| : f \mapsto \sup_{t \in [0,1]}\| f(t) \|_1$.

By Lemma \ref{l:traconint} we have that the path
\[
t \mapsto \mu^{-1/2} \cd 
\big( R_t^{1 + \mu} \cd X_t^{1 + \mu} + (X_t^{1 + \mu})^* \cd R_t^{1 + \mu} \big) \cd \C D_t \cd (R_0)^n
\]
determines an element in $C\big( [0,1], \sL^1(\C H \op \C H) \big)$ for all $\mu > 0$. It is therefore enough to show that the integral
\[
\int_0^\infty \mu^{-1/2} \cd \sup_{t \in [0,1]} \big\| 
\big( R_t^{1 + \mu} \cd X_t^{1 + \mu} + (X_t^{1 + \mu})^* \cd R_t^{1 + \mu} \big) \cd \C D_t \cd (R_0)^n \big\|_1 \, d\mu
\]
is finite. But this is also a consequence of Lemma \ref{l:traconint}.
\end{proof}

{\color{black}
We now prove the main result of this paper. It gives sufficient conditions for the invariance of the {\color{black} unbounded} homological index under certain {\color{black} unbounded} perturbations. The proof relies on Theorem \ref{t:hominvII} which in turn is a consequence of our investigations of homotopy invariance in cyclic theory in Section \ref{s:invarpro}.
}

\begin{thm}\label{t:hominvunb}
Suppose that $\C D^+ : \T{Dom}(\C D^+) \to \C H$ and $B^+ : \T{Dom}(B^+) \to \C H$ are densely defined closed unbounded operators which satisfy the conditions in Assumption \ref{a:patrelbou} and Assumption \ref{a:sumpat} for some $n \in \nn$.

Then the homological index of $\C D^+ + B^+$ exists in degree $n \in \nn$ if and only if the homological index of $\C D^+$ exists in degree $n \in \nn$. And in this case we have that
\[
\T{H-Ind}_n(\C D^+) = \T{H-Ind}_n(\C D^+ + B^+)
\]
\end{thm}
\begin{proof}
{\color{black} This is a consequence of Theorem \ref{t:exihomind}, Theorem \ref{t:hominvII}, Proposition \ref{p:traconder} and Proposition \ref{p:consqu}. See the discussion in the beginning of this Section.}
\end{proof}

\section{Appendix I: Dirac operators on Euclidean space}
Let $n \in \nn$ and consider the complex Clifford algebra $\cc_{2n-1}$ over $\rr^{2n-1}$. Recall that this is the unital $*$-algebra with $(2n-1)$ generators $e_1,\ldots,e_{2n-1}$ such that
\[
\arr{ccc}{
e_j = e_j^* & \T{and} & e_j e_k + e_k e_j = 2 \de_{jk}
}
\]
for all $j,k \in \{1,\ldots,2n-1\}$, where $\de_{jk}$ denotes the Kronecker delta. Fix an irreducible representation $\pi_{2n-1} : \cc_{2n-1} \to \sL(\cc^{2^{n-1}})$ and apply the notation $c_j := \pi_{2n-1}(e_j)$ for all $j \in \{1,\ldots,2n-1\}$.

Let $N \in \nn$ and let $a_j : \rr^{2n} \to M_N(\cc)$ be a matrix valued function for each $j \in \{1,\ldots,2n\}$. Suppose that $a_j : \rr^{2n} \to M_N(\cc)$ is bounded and measurable with respect to Lebesgue measure and furthermore that $a_j(x) = a_j(x)^*$ for all $x \in \rr^{2n}$ and all $j \in \{1,\ldots,2n\}$.
We may then form the closed unbounded operator
\[
\C D^+_a := \frac{\pa}{\pa x_{2n}} + i a_{2n} + i \sum_{j=1}^{2n-1} c_j \cd \big( \frac{\pa}{\pa x_j} + i a_j\big)
\]
which acts on the Hilbert space $L^2(\rr^{2n}) \ot \cc^{2^{n-1}} \ot \cc^N$. The domain of $\C D^+_a$ is the first Sobolev space $H^1(\rr^{2n}) \ot \cc^{2^{n-1}} \ot \cc^N$.
The (Hilbert space) adjoint of $\C D^+_a$ is the first order differential operator
\[
\C D^-_a := (\C D^+_a)^* = 
- \frac{\pa}{\pa x_{2n}} - i a_{2n} + i \sum_{j=1}^{2n-1} c_j \cd \big( \frac{\pa}{\pa x_j} + i a_j\big)
\]
The domain of $\C D^-_a$ is again the first Sobolev space $H^1(\rr^{2n}) \ot \cc^{2^{n-1}} \ot \cc^N \su L^2(\rr^{2n}) \ot \cc^{2^{n-1}} \ot \cc^N$.
We refer to the selfadjoint unbounded operator
\[
\begin{split}
\C D_a & := \ma{cc}{ 0 & \C D^-_a \\ \C D^+_a & 0}
: \big( H^1(\rr^{2n}) \ot \cc^{2^{n-1}} \ot \cc^N \big) \op \big(H^1(\rr^{2n}) \ot \cc^{2^{n-1}} \ot \cc^N\big) \\
& \q \to 
\big( L^2(\rr^{2n}) \ot \cc^{2^{n-1}} \ot \cc^N \big)
\op \big( L^2(\rr^{2n}) \ot \cc^{2^{n-1}} \ot \cc^N \big)
\end{split}
\]
as the \emph{generalized Dirac operator} associated with the bounded measurable matrix valued one-form $\sum_{j=1}^{2n} a_j dx_j$.
The following vanishing result is a consequence of our general invariance result for the homological index, see Theorem \ref{t:hominvunb}.

\begin{thm}\label{t:homindvan}
Suppose that there exists a $\de > 0$ such that the map
\[
(1 + \sum_{k = 1}^{2n} x_k^2)^{n/2 + \de} \cd a_j : \rr^{2n} \to M_N(\cc)
\]
is square integrable for each $j \in \{1,\ldots,2n\}$. Then the homological index of $\la^{-1/2} \cd \C D_a$ in degree $n$ exists and is equal to zero for all $\la > 0$.
\end{thm}
\begin{proof}
To ease the notation, let $\C H := L^2(\rr^{2n}) \ot \cc^{2^{n-1}} \ot \cc^N$.
Consider the generalized Dirac operator associated with the trivial matrix valued one-form,
\[
\C D_0 := \ma{cc}{0 & \C D_0^- \\ \C D_0^+ & 0}
\]
Remark that $\C D_0^+ : H^1(\rr^{2n}) \ot \cc^{2^{n-1}} \ot \cc^N \to \C H$ is normal with
\[
\C D_0^- \C D_0^+ = -\sum_{j =1}^{2n} \frac{\pa^2}{\pa x_j^2} = \C D_0^+ \C D_0^-
\]
The domain of the Laplacian $\lap := -\sum_{j =1}^{2n} \frac{\pa^2}{\pa x_j^2}$ is the subspace $H^2(\rr^{2n}) \ot \cc^{2^{n-1}} \ot \cc^N$ where $H^2(\rr^{2n})$ denotes the second Sobolev space.

The normality of $\C D_0^+$ implies that the homological index of $\la^{-1/2} \C D_0^+$ exists and is trivial in all degrees for all $\la > 0$.
Notice now that the generalized Dirac operator $\C D_a^+$ is a bounded perturbation of $\C D_0^+$. Indeed, we have that
\[
\C D_a^+ = \C D_0^+ + i \cd (1 \ot a_{2n}) - c_j \ot a_j
\]
where the bounded operator
\[
A^+ := i \cd (1 \ot  a_{2n}) - c_j \ot a_j : \C H \to \C H
\]
acts by pointwise matrix multiplication.

By Theorem \ref{t:hominvunb} it is therefore enough to show that the closed unbounded operator $\C D_0^+ : \T{Dom}(\C D_0^+) \to \C H$ and the bounded operator $A^+ \in \sL( \C H ) $ satisfy the conditions in Assumption \ref{a:patrelbou} and Assumption \ref{a:sumpat}.
Since $A^+$ is a bounded operator, the conditions in Assumption \ref{a:patrelmai} can be verified immediately. Furthermore, it follows from \cite[Proposition 2.10]{CaPh:UFS} that the path of bounded transform
$
t \mapsto \C D_{t \cd a} \cd \big( 1 + \C D_{t \cd a}^2 \big)^{-1/2}
$
is continuously differentiable in operator norm. We have thus verified the conditions in Assumption \ref{a:patrelbou}.

It remains to verify that
\begin{equation}\label{eq:intpowsch}
(1 + \lap)^{-j} \cd A^+ \cd (1 + \lap)^{-k-1/2} \in \sL^{n/(j+k)}(\C H)
\end{equation}
for all $j,k \in \{0,\ldots,n\}$ with $1 \leq j + k \leq n$ and that
\begin{equation}\label{eq:frapowsch}
A^+ \cd (1 + \lap)^{-n - 1/2 + \ep} \, \, , \, \, (A^+)^* \cd (1 + \lap)^{-n - 1/2 + \ep}  \in \sL^1(\C H)
\end{equation}
for some $\ep \in (0,1/2)$. See Assumption \ref{a:sumpat}.

{\color{black}
Using our conditions on the bounded measurable maps $a_j : \rr^{2n} \to M_N(\cc)$ we see that \eqref{eq:intpowsch} and \eqref{eq:frapowsch} are consequences of the next two Lemmas. Their proofs will conclude the proof of the Theorem.
}
\end{proof}

\begin{lemma}\label{l:firsch}
Let $f : \rr^m \to \cc$ be a bounded measurable function and let $p \in [1,\infty)$. Suppose that the function
\[
(1 + \sum_{k=1}^m x_k^2)^{\frac{m}{4p} + \de} \cd f : \rr^m \to \cc
\]
lies in $L^{2p}(\rr^m)$ for some $\de > 0$. Then the bounded operator
\[
f \cd (1 + \lap)^{-\frac{m}{2p} - \ep} : L^2(\rr^m) \to L^2(\rr^m)
\]
lies in the Schatten ideal $\sL^p( L^2(\rr^m))$ for all $\ep > 0$.
\end{lemma}
\begin{proof}
Suppose first that $p \geq 2$. By \cite[Theorem 4.1]{Sim:TIA} it is enough to show that $f \in L^p(\rr^m)$ and that $(1 + \sum_{k=1}^m x_k^2)^{-\frac{m}{2p} - \ep} \in L^p(\rr^m)$ for all $\ep > 0$. It is clear that $(1 + \sum_{k=1}^m x_k^2)^{-\frac{m}{2p} - \ep} \in L^p(\rr^m)$ for all $\ep > 0$ thus we only need to show that $f \in L^p(\rr^m)$.

It follows from our assumptions that there exists a positive integrable function $g : \rr^m \to [0,\infty)$ such that
\[
g \cd (1 + \sum_{k = 1}^m x_k^2)^{-\frac{m}{2} - \de_0} = |f|^{2p}
\]
for some $\de_0 > 0$. This proves that $|f|^p$ is integrable since it is the product of the two square integrable functions $g^{1/2}$ and $(1 + \sum_{k = 1}^m x_k^2)^{-\frac{m}{4} - \frac{\de_0}{2}}$.

Suppose now that $p \in [1,2]$. By \cite[Theorem 4.5]{Sim:TIA} it is then enough to show that $f \in l^p(L^2)$ and that $(1 + \sum_{k = 1}^m x_k^2)^{-\frac{m}{2p} - \ep} \in l^p(L^2)$ for all $\ep > 0$. Recall here that the space $l^p(L^2)$ consists of the measurable functions $g : \rr^m \to \cc$ for which
\[
\| g \|_{2;p} := \Big( \sum_{\al \in \zz^m} \big( \int_{\De_\al} |g|^2 dx \big)^{p/2} \Big)^{1/p} 
\]
is finite, where $\De_\al \su \rr^m$ is the unit cube with center in $\al \in \zz^m$.

Since the function $(1 + \sum_{k = 1}^m x_k^2)^{-\frac{m}{2p} - \ep} : \rr^m \to [0,\infty)$ satisfies the same hypothesis as $f : \rr^m \to [0,\infty)$ it is enough to show that $f \in l^p(L^2)$. To this end, we note that
\[
\begin{split}
\big( \int_{\De_\al} |f|^2 dx \big)^{p/2} 
& \leq \big(1 + \sum_{k=1}^m (\al_j/2)^2\big)^{-\frac{m}{4} - \de \cd p} \cd 
\big( \int_{\De_\al} (1 + \sum_{k=1}^m x_k^2)^{\frac{m}{2p} + 2 \cd \de} |f|^2 \, dx \big)^{p/2} \\
& \leq \big(1 + \sum_{k=1}^m (\al_j/2)^2\big)^{-\frac{m}{4} - \de \cd p}
\cd \big( \int_{\De_\al} (1 + \sum_{k=1}^m x_k^2)^{\frac{m}{2} + 2p \cd \de} |f|^{2p} \, dx \big)^{1/2}
\end{split}
\]
for all $\al \in \zz^m$. This estimate implies that $f \in l^p(L^2)$ since both of the sequences
\[
 \al \mapsto \big(1 + \sum_{k=1}^m (\al_j/2)^2\big)^{-\frac{m}{4} - \de \cd p} \q \T{and} \q
\al \mapsto \big( \int_{\De_\al} (1 + \sum_{k=1}^m x_k^2)^{\frac{m}{2} + 2p \cd \de} |f|^{2p} \, dx \big)^{1/2}
\]
are square summable.
\end{proof}

\begin{lemma}\label{l:secsch}
Let $f : \rr^m \to \cc$ be a bounded measurable function. Suppose that there exists a $\de > 0$ such that
$
(1 + \sum_{k=1}^m x_k^2)^{m/4 + \de} \cd f : \rr^m \to \cc
$
is square integrable. Then
\[
(1 + \lap)^{-s} \cd f \cd (1 + \lap)^{-r-\ep} \in \sL^{\frac{m}{2(s+r)}}\big( L^2(\rr^m)\big)
\]
for all $s,r \geq 0$ with $s + r \in (0,m/2]$ and all $\ep > 0$. 
\end{lemma}
\begin{proof}
Let $s,r \geq 0$ with $s + r \in (0, m/2]$ and $\ep > 0$ be fixed.
Suppose first that $s \in [0,\ep/2]$. By Lemma \ref{l:firsch} it is then enough to show that the measurable function
\[
(1 + \sum_{k=1}^m x_k^2)^{\frac{r + s}{2} + \de_0} \cd f : \rr^m \to \cc
\]
lies in $L^{\frac{m}{r+s}}(\rr^m)$ for some $\de_0 > 0$. Or in other words that
\[
(1 + \sum_{k = 1}^m x_k^2)^{\frac{m}{4} + \de} \cd |f|^{\frac{m}{2(r+s)}} : \rr^m \to [0,\infty) 
\]
is square integrable for some $\de > 0$. But this is a consequence of our asssumptions since $\frac{m}{2(r+s)} \geq 1$ and $f : \rr^m \to \cc$ is bounded.

Suppose now that $s > \ep/2$. Consider the polar decomposition $f = u \cd |f|$. Thus, $u : \rr^m \to \cc$ is a measurable function with $|u|=1$. By the H\"older inequality it is then enough to show that
\[
\begin{split}
& |f|^{\frac{s - \ep/2}{s + r}} \cd (1 + \lap)^{-s} \in \sL^{\frac{m}{2(s-\ep/2)}}\big( L^2(\rr^m)\big)\q \T{and} \\
& |f|^{\frac{r + \ep/2}{s + r}} \cd (1 + \lap)^{-r - \ep} \in \sL^{\frac{m}{2(r + \ep/2)}}\big( L^2(\rr^m)\big)
\end{split}
\]
For symmetry reasons, we may restrict our attention to the first of these bounded operators.

By Lemma \ref{l:firsch} it suffices to prove that the measurable function
\[
(1 + \sum_{k = 1}^m x_k^2)^{\frac{s - \ep/2}{2} + \de_0} \cd |f|^{\frac{s - \ep/2}{s + r}} : \rr^m \to \cc
\]
lies in $L^{\frac{m}{s - \ep/2}}(\rr^m)$ for some $\de_0 > 0$. Or equivalently that
\[
(1 + \sum_{k = 1}^m x_k^2)^{\frac{m}{4} + \de} \cd |f|^{\frac{m}{2(s + r)}} : \rr^m \to \cc
\]
is square integrable for some $\de > 0$. But this follows from our assumptions since $\frac{m}{2(s + r)} \geq 1$ and $f : \rr^m \to \cc$ is bounded.

\end{proof}


\section{Appendix II: Perturbations of unbounded operators}
Throughout this Appendix $\C D^+ : \T{Dom}(\C D^+) \to \C H$ and $A^+ : \T{Dom}(A^+) \to \C H$ will be two \emph{closed} densely defined unbounded operators. The Hilbert space adjoints of $\C D^+$ and $A^+$ will be denoted by $\C D^- := (\C D^+)^*$ and $A^- := (A^+)^*$. 
%

We remark that
\[
\begin{split}
& \C D := \ma{cc}{0 & \C D^- \\ \C D^+ & 0} : \T{Dom}(\C D^+) \op \T{Dom}(\C D^-) \to \C H \op \C H \q \T{and} \\
& A := \ma{cc}{0 & A^- \\ A^+ & 0} : \T{Dom}(A^+) \op \T{Dom}(A^-) \to \C H \op \C H
\end{split}
\]
are selfadjoint unbounded operators.

\begin{lemma}\label{l:clopat}
Suppose that $\T{Dom}(\C D^+) \su \T{Dom}(A^+)$ and that $\T{Dom}(\C D^-) \su \T{Dom}(A^-)$. Then the unbounded operator
$
\C D^+ + t \cd A^+ : \T{Dom}(\C D^+) \to \C H
$
is closable for each $t \in [0,1]$.
\end{lemma}
\begin{proof}
Let $t \in [0,1]$. By \cite[Theorem VIII.1]{ReSi:MMPI} it is enough to show that $\C D^+ + t \cd A^+ : \T{Dom}(\C D^+) \to \C H$ has a densely defined adjoint. But this is immediate since
\[
\inn{\xi , (\C D^+ + t \cd A^+)\eta} = \inn{(\C D^- + t \cd A^-)\xi, \eta}
\]
for all $\xi \in \T{Dom}(\C D^-)$ and all $\eta \in \T{Dom}(\C D^+)$.
\end{proof}

With the conditions of Lemma \ref{l:clopat} we apply the notation
\[
\C D_t^+ := \ov{\C D^+ + t \cd A} : \T{Dom}(\C D_t^+) \to \C H \q \T{and} \q
\C D_t^- := (\C D_t^+)^* : \T{Dom}(\C D_t^-) \to \C H
\]
for the closure of $\C D^+ + t \cd A : \T{Dom}(\C D^+) \to \C H$ and its adjoint.
We remark that the unbounded operator
\[
\C D_t := \ma{cc}{ 0 & \C D_t^- \\ \C D_t^+ & 0} : \T{Dom}(\C D^+_t) \op \T{Dom}(\C D^-_t) \to \C H \op \C H
\]
is selfadjoint for all $t \in [0,1]$.

\begin{assu}\label{a:patrelbouI}
Suppose that the following holds:
\begin{enumerate}
\item $\T{Dom}(\C D^+) \su \T{Dom}(A^+)$ and $\T{Dom}(\C D^-) \su \T{Dom}(A^-)$.
\item There exists a dense subspace $\C E \su \C H \op \C H$ such that $(i + \C D_t)^{-1}(\xi) \in \T{Dom}(A)$ for all $\xi \in \C E$ and all $t \in [0,1]$.
\item The unbounded operator $A \cd (i + \C D_t)^{-1} : \C E \to \C H \op \C H$ extends to a bounded operator $X_t : \C H \op \C H \to \C H \op \C H$ for all $t \in [0,1]$.
\end{enumerate}
\end{assu}

\emph{Unless explicitly mentioned, the conditions of Assumption \ref{a:patrelbouI} will remain in effect throughout this Appendix.}

\begin{lemma}
We have that $\T{Dom}(\C D_t) \su \T{Dom}(A)$ for all $t \in [0,1]$.
\end{lemma}
\begin{proof}
Let $t \in [0,1]$ and let $\xi \in \C H \op \C H$. 
Since $\C E \su \C H \op \C H$ is a dense subspace we may choose a sequence $\{\xi_n\}$ in $\C E$ which converges to $\xi$.
We then have that
\[
(i + \C D_t)^{-1}(\xi_n) \to (i + \C D_t)^{-1}\xi \q \T{and} \q X_t(\xi_n) \to X_t(\xi)
\]
But this implies that $(i + \C D_t)^{-1}(\xi) \in \T{Dom}(A)$ since $A$ is closed and since $X_t(\xi_n) = A \cd (i + \C D_t)^{-1}(\xi_n)$ for all $n \in \nn$.
\end{proof}

For each $\la > 0$ and each $t \in [0,1]$ we introduce the bounded operators
\[
X_t^\la := A \cd (i \cd \la^{1/2} + \C D_t)^{-1} \, \T{ and } \, R_t^\la := (\la + \C D_t^2)^{-1} : \C H \op \C H \to \C H \op \C H
\]
The next Lemma is crucial:

\begin{lemma}\label{l:resideI}
We have the identities
\begin{equation}\label{eq:resideI}
\begin{split}
(i \cd \la^{1/2} + \C D_t)^{-1} - (i \cd \la^{1/2} + \C D)^{-1} 
& = - t \cd (i \cd \la^{1/2} + \C D_t)^{-1} \cd X_0^\la \\
& = - t \cd (i \cd \la^{1/2} + \C D)^{-1} \cd X_t^\la
\end{split}
\end{equation}
for all $t \in [0,1]$ and all $\la > 0$.
\end{lemma}
\begin{proof}
Let $t \in [0,1]$ and let $\la > 0$.We 
remark first that it follows from the proof of Lemma \ref{l:clopat} that $\T{Dom}(\C D) \su \T{Dom}(\C D_t)$. Furthermore, we have that $\C D_t(\xi) = (\C D + t \cd A)(\xi)$ for all $\xi \in \T{Dom}(\C D)$.
The first identity in \eqref{eq:resideI} can now be verified immediately.

A similar argument shows that
\[
(-i \cd \la^{1/2} + \C D_t)^{-1} - (-i \cd \la^{1/2} + \C D)^{-1}
= - t \cd (-i \cd \la^{1/2} + \C D_t)^{-1} \cd A \cd (-i \cd \la^{1/2} + \C D)^{-1}
\]
The second identity in \eqref{eq:resideI} now follows by taking adjoints.
\end{proof}

The power of Lemma \ref{l:resideI} is illustrated in the next Proposition.

\begin{prop}\label{p:dompat}
Suppose that $A^+ : \T{Dom}(A^+) \to \C H$ and $\C D^+ : \T{Dom}(\C D^+)$ satisfy the conditions of Assumption \ref{a:patrelbouI}. Then we have that $\C D^+ + t \cd A^+ : \T{Dom}(\C D^+) \to \C H$ is closed and the adjoint is given by $(\C D^+ + t \cd A^+)^* = \C D^- + t \cd A^- : \T{Dom}(\C D^-) \to \C H$ for all $t \in [0,1]$.
\end{prop}
\begin{proof}
Let $t \in [0,1]$. It is enough to show that $\T{Dom}(\C D) = \T{Dom}(\C D_t)$. 
By the proof of Lemma \ref{l:clopat} we have that $\T{Dom}(\C D) \su \T{Dom}(\C D_t)$.
Thus, let $\xi \in \T{Dom}(\C D_t)$. By Lemma \ref{l:resideI} we have that
\begin{equation}\label{eq:dompat}
\xi = (i + \C D_t)^{-1} \cd (i + \C D_t)(\xi) = (i + \C D)^{-1} \cd (i + \C D_t)(\xi)
- t \cd (i + \C D)^{-1} \cd X_t(\xi)
\end{equation}
This shows that $\xi \in \T{Dom}(\C D)$ since both of the terms in \eqref{eq:dompat} lie in $\T{Dom}(\C D)$.
\end{proof}

The resolvent identity in Lemma \ref{l:resideI} can now be extended to the whole path:

\begin{lemma}\label{l:resideII}
We have the identity
\[
(i \cd \la^{1/2} + \C D_t)^{-1} - (i \cd \la^{1/2} + \C D_s)^{-1} 
= (s - t) \cd (i \cd \la^{1/2} + \C D_t)^{-1} \cd X_s^\la
\]
for all $\la > 0$ and all $t,s \in [0,1]$.
\end{lemma}
\begin{proof}
Let $t,s \in [0,1]$ and let $\la > 0$. By Proposition \ref{p:dompat}, we have that $\T{Dom}(\C D) = \T{Dom}(\C D_s) = \T{Dom}(\C D_t)$. Furthermore, we have the identity $(\C D_s - \C D_t)(\xi) = (s - t) \cd A(\xi)$ for all $\xi \in \T{Dom}(\C D)$. The result of the Lemma now follows by a straightforward computation.
\end{proof}

As an easy consequence of the resolvent identity in Lemma \ref{l:resideII} we obtain the following Lemma. The proof is left to the reader.

\begin{lemma}\label{l:resideIII}
We have the identities
\[
\begin{split}
X_t^\la - X_s^\la & = (s-t) \cd X_t^\la \cd X_s^\la \q \T{and} \\
R_t^\la - R_s^\la & = (s-t) \cd \big( R_t^\la \cd X_s^\la + (X_t^\la)^* \cd R_s^\la \big)
\end{split}
\]
for all $\la > 0$ and all $t,s \in [0,1]$.
\end{lemma}
%
%

\emph{For the rest of this Appendix we will need an extra assumption on our data:}

\begin{assu}\label{a:patrelbouII}
Suppose that $A^+ : \T{Dom}(A^+) \to \C H$ and $\C D^+ : \T{Dom}(\C D^+)$ satisfy the conditions of Assumption \ref{a:patrelbouI}. Suppose furthermore that $\sup_{t \in [0,1]}\| X_t \| < \infty$.
\end{assu}

 Lemma \ref{l:resideIII} now entails the following result though we remark that our extra assumption on the uniform boundedness of the path $t \mapsto X_t$ is needed here.
\begin{lemma}\label{l:derres}
Let $\la > 0$. The paths $t \mapsto R_t^\la$ and $t \mapsto X_t^\la$ are continuously differentiable in operator norm. The derivatives are given by
\[
\frac{d (X_t^\la)}{dt}\Big|_{t_0} = - (X_{t_0}^\la)^2 \q \T{and} \q
\frac{d(R_t^\la)}{dt}\Big|_{t_0} = - R_{t_0}^\la \cd X_{t_0}^\la - (X_{t_0}^\la)^* \cd R_{t_0}^\la
\]
for all $t_0 \in [0,1]$.
\end{lemma}

The next result is now a consequence of the Leibniz rule:

\begin{prop}\label{p:derpowres}
Let $\la > 0$ and let $n \in \nn$. The path $t \mapsto (R_t^\la)^n$ is continuously differentiable in operator norm and the derivative is given by
\[
\frac{d\big( (R_t^\la)^n\big)}{dt}\Big|_{t_0} = -\sum_{j = 0}^{n-1} (R_{t_0}^\la)^j \cd 
\big( R_{t_0}^\la \cd X_{t_0}^\la + (X_{t_0}^\la)^* \cd R_{t_0}^\la \big) \cd (R_{t_0}^\la)^{n-1-j}
\]
for all $t_0 \in [0,1]$.
\end{prop}

For each $t \in [0,1]$ we let
$
T_t := \C D_t \cd \big(1 + \C D_t^2 \big)^{-1/2}
$
denote the bounded transform of $\C D_t : \T{Dom}(\C D) \to \C H$.

\begin{prop}\label{p:derboutra}
Suppose that $A^+ : \T{Dom}(A^+) \to \C H$ and $\C D^+ : \T{Dom}(\C D^+)$ satisfy the conditions of Assumption \ref{a:patrelbouII}. Then the path $t \mapsto T_t \cd (i + \C D)^{-1}$ is continuously differentiable and the derivative is given by
\[
\begin{split}
& \frac{d\big(T_t \cd (i + \C D)^{-1}\big)}{dt}\Big|_{t_0} \\ 
& \q = (1 + \C D_{t_0}^2)^{-1/2} \cd X_0 \\
& \qqq - \frac{1}{\pi} \cd \int_0^\infty \mu^{-1/2} \cd 
\big( R_{t_0}^{1 + \mu} \cd X_{t_0}^{1 + \mu} + (X_{t_0}^{1 + \mu})^* \cd R_{t_0}^{1 + \mu} \big)
\, d\mu \cd \C D_{t_0} \cd (i + \C D)^{-1}
\end{split}
\]
for all $t_0 \in [0,1]$, where the integral converges absolutely in operator norm.
\end{prop}
\begin{proof}
For each $t \in [0,1]$ we have the identity
\[
T_t \cd (i + \C D)^{-1} = (1 + \C D_t)^{-1/2} \cd \C D_t \cd (i + \C D)^{-1}
\]

The path $t \mapsto \C D_t \cd (i + \C D)^{-1}$ is clearly continuously differentiable and the derivative is given by
\[
\frac{d\big(\C D_t \cd (i + \C D)^{-1} \big)}{dt}\big|_{t_0} = A \cd (i + \C D)^{-1} = X_0
\]
for all $t_0 \in [0,1]$.
The result of our Proposition is therefore a consequence of the next Lemma and the Leibniz rule.
\end{proof}

\begin{lemma}\label{l:dersqu}
Suppose that $A^+ : \T{Dom}(A^+) \to \C H$ and $\C D^+ : \T{Dom}(\C D^+)$ satisfy the conditions of Assumption \ref{a:patrelbouII}. Then the path
$t \mapsto (1 + \C D_t^2)^{-1/2}$ is continuously differentiable and the derivative is given by
\[
\frac{d\big( (1 + \C D_t^2)^{-1/2} \big)}{dt}\Big|_{t_0} = - \frac{1}{\pi} 
\cd \int_0^\infty \mu^{-1/2} \cd \big( R_{t_0}^{1 + \mu} \cd X_{t_0}^{1 + \mu} + (X_{t_0}^{1 + \mu})^* \cd R_{t_0}^{1 + \mu} \big) \, d\mu
\]
for all $t_0 \in [0,1]$, where the integral converges absolutely in operator norm.
\end{lemma}
\begin{proof}
For each $t \in [0,1]$, we may express the bounded operator $(1 + \C D_t^2)^{-1/2}$ by an integral formula,
\[
(1 + \C D_t^2)^{-1/2} 
= \frac{1}{\pi} \cd \int_0^\infty \mu^{-1/2} \cd (1 + \mu + \C D_t^2)^{-1} \, d\mu
= \frac{1}{\pi} \cd \int_0^\infty \mu^{-1/2} \cd R_t^{1 + \mu} \, d\mu
\]
where the integral converges absolutely in operator norm.

Let us apply the notation for the unital algebra $C^1\big( [0,1], \sL(\C H \op \C H) \big)$ consisting of all maps $[0,1] \to \sL(\C H \op \C H)$ which are continuously differentiable in operator norm. This unital algebra becomes a Banach algebra when equipped with the norm
\[
\| \cd \|_1 : C^1\big( [0,1], \sL(\C H \op \C H) \big) \to [0,\infty)
\q \| f \|_1 := \sup_{t_0 \in [0,1]}\| f(t_0) \| + \sup_{t_0 \in [0,1]} \| \frac{df}{dt}\big|_{t_0} \|
\]

Recall now from Lemma \ref{l:derres} that the map $t \mapsto R_t^{1 + \mu}$ lies in $C^1\big( [0,1], \sL(\C H \op \C H) \big)$ for all $\mu \in [0,\infty)$. Furthermore, we have that 
\[
\frac{d R_t^{1 + \mu}}{dt}\Big|_{t_0} = - R_{t_0}^{1 + \mu} \cd X_{t_0}^{1 + \mu} - (X_{t_0}^{1 + \mu})^* \cd R_{t_0}^{1 + \mu}
\]
for all $t_0 \in [0,1]$.

The result of the present Lemma therefore follows by noting that both of the integrals
\[
\begin{split}
& \int_0^\infty \mu^{-1/2} \cd \sup_{t_0 \in [0,1]} \| R_{t_0}^{1 + \mu} \| \, d\mu \q \T{and} \\
& \int_0^\infty \mu^{-1/2} \cd \sup_{t_0 \in [0,1]} \| R_{t_0}^{1 + \mu} \cd X_{t_0}^{1 + \mu} + (X_{t_0}^{1 + \mu})^* \cd R_{t_0}^{1 + \mu}\| \, d\mu
\end{split}
\]
are finite.
\end{proof}

\newcommand{\etalchar}[1]{$^{#1}$}
\def\cprime{$'$}
\providecommand{\bysame}{\leavevmode\hbox to3em{\hrulefill}\thinspace}
\providecommand{\MR}{\relax\ifhmode\unskip\space\fi MR }
\providecommand{\MRhref}[2]{%
  \href{http://www.ams.org/mathscinet-getitem?mr=#1}{#2}
}
\providecommand{\href}[2]{#2}

\bibliographystyle{amsalpha-lmp}

\end{document}